\documentclass{amsart}%
\usepackage{amsfonts}
\usepackage{amsmath}
\usepackage{amssymb}
\usepackage{graphicx}%
\newtheorem{theorem}{Theorem}
\theoremstyle{plain}

\newtheorem{definition}{Definition}
\newtheorem{example}{Example}

\newtheorem{lemma}{Lemma}

\newtheorem{remark}{Remark}

\numberwithin{equation}{section}
\ifx\pdfoutput\relax\let\pdfoutput=\undefined\fi
\newcount\msipdfoutput
\ifx\pdfoutput\undefined\else
\ifcase\pdfoutput\else
\ifx\paperwidth\undefined\else
\ifdim\paperheight=0pt\relax\else\pdfpageheight\paperheight\fi
\ifdim\paperwidth=0pt\relax\else\pdfpagewidth\paperwidth\fi
\fi\fi\fi
\begin{document}
\title[Pseudodifferential operators and Feller Semigroups]{Non-Archimedean pseudodifferential operators\ and Feller Semigroups}
\author{Anselmo Torresblanca-Badillo}
\address{Universidad del Norte, Departamento de Matem\'{a}ticas y Est\'{a}distica, Km.
5 V\'{\i}a Puerto Colombia, Barranquilla, Colombia.}
\email{atorresblanca@uninorte.edu.co}
\author{W. A. Z\'{u}\~{n}iga-Galindo}
\address{Centro de Investigaci\'{o}n y de Estudios Avanzados del Instituto
Polit\'{e}cnico Nacional\\
Departamento de Matem\'{a}ticas, Unidad Quer\'{e}taro\\
Libramiento Norponiente \#2000, Fracc. Real de Juriquilla. Santiago de
Quer\'{e}taro, Qro. 76230\\
M\'{e}xico. }
\email{wazuniga@math.cinvestav.edu.mx}
\thanks{The second author was partially supported by Conacyt Grant No. 250845.}

\begin{abstract}
In this article we study a class of \ non-Archimedean pseudodifferential
operators whose symbols are negative definite functions. We prove that these
operators extend to generators of Feller semigroups. In order to study these
operators, we introduce a new class of anisotropic Sobolev spaces, which are
the natural domains for the operators considered here. We also study the
Cauchy problem for certain pseudodifferential equations.

\end{abstract}
\maketitle

\section{Introduction}

The interplay between pseudodifferential operators and stochastic processes
constitutes a classical area of research in the Archimedean setting see e.g.
\cite{Jacob-vol-1}-\cite{Jacob-vol-3}, \cite{Taira} and the references
therein, and in the non-Archimedean one, see e.g. \cite{Koch}, \cite{Kozyrev
SV}, \cite{Va1}, \cite{V-V-Z}, \cite{Zuniga-LNM-2016}, \cite{Zu1}, \cite{Z1},
and the references therein.

This article aims to study a class of non-Archimedean pseudodifferential
operators having negative definite symbols, which have attached Feller
semigroups. These operators have the form
\[
\left(  A\left(  \partial\right)  \varphi\right)  \left(  x\right)
=-\mathcal{F}_{\xi\rightarrow x}^{-1}\left(  a\left(  \xi\right)
\mathcal{F}_{x\rightarrow\xi}\varphi\right)
\]
where $\varphi:%
\mathbb{Q}
_{p}^{n}\rightarrow%
\mathbb{C}
$ is a test function, $%
\mathbb{Q}
_{p}$ is the field of $p$-adic numbers, $\mathcal{F}_{x\rightarrow\xi}$
denotes the Fourier transform, and the symbol $\xi\rightarrow a\left(
\xi\right)  $ is a negative definite function, for $x\in%
\mathbb{Q}
_{p}^{n}$ and $t\in\mathbb{R}_{+}$. A typical example of such symbols are
functions of the form $\sum_{j=1}^{m}b_{j}\boldsymbol{\psi}_{j}\left(
\xi\right)  $, where the $\boldsymbol{\psi}_{j}:\mathbf{%
\mathbb{Q}
}_{p}^{n}\rightarrow%
\mathbb{R}
_{+}$ are radial (i.e. $\boldsymbol{\psi}_{j}\left(  \xi\right)
=\boldsymbol{\psi}_{j}\left(  ||\xi||_{p}\right)  $), continuous, negative
definite functions and the coefficients $b_{j}$ are positive real numbers. In
the non-Archimedean setting there exist `exotic' negative definite functions
such as $\exp\left(  \exp\left(  \exp\sum_{j=0}^{\infty}a_{j}\left\Vert
\xi\right\Vert _{p}^{\alpha_{j}}\right)  \right)  $ where $\sum_{j=0}^{\infty
}a_{j}y^{\alpha_{j}}$ is a convergent real series with $a_{j}>0$, $\alpha
_{j}>0$ and $\lim_{j\rightarrow\infty}\alpha_{j}=\infty$, see Lemma
\ref{exp Fi negative definite}. This type of functions do not have Archimedean counterparts.

Let $P(\partial)$ denote a pseudodifferential operator whose symbol is a
negative definite function $p\left(  \xi\right)  $. We introduce a new class
of function spaces $B_{\boldsymbol{\psi},\infty}\left(  \mathbb{R}\right)  $
attached to a negative definite function $\boldsymbol{\psi}$ related to
$p\left(  \xi\right)  $. These spaces are generalizations of the spaces
$H_{\infty}\left(  \mathbb{R}\right)  $ introduced by the second author in
\cite{Zu1}. The spaces $B_{\boldsymbol{\psi},\infty}\left(  \mathbb{R}\right)
$ are nuclear countably Hilbert spaces. We show that $B_{\boldsymbol{\psi
},\infty}\left(  \mathbb{R}\right)  $ is the natural domain for an operator of
type $P(\partial)$ for a suitable $\boldsymbol{\psi}$. Under mild hypotheses,
we show that $\left(  P(\partial),B_{\boldsymbol{\psi},\infty}\left(
\mathbb{R}\right)  \right)  $ has a closed extension to $C_{0}\left(
\mathbb{Q}
_{p}^{n},\mathbb{R}\right)  $ (the $\mathbb{R}$-vector space of bounded
continuous functions vanishing at infinity) which is the generator of a Feller
semigroup, see Theorem \ref{Feller semigroups}. We also study the following
Cauchy problem:
\begin{equation}
\left\{
\begin{array}
[c]{ll}%
\frac{\partial u}{\partial t}(x,t)=P(\partial)u(x,t)+f(x,t)\text{,} &
t\in\lbrack0,T]\text{,\ }x\in%
\mathbb{Q}
_{p}^{n};\\
& \\
u(x,0)=h(x)\in B_{\boldsymbol{\psi},\infty}\left(  \mathbb{R}\right)  . &
\end{array}
\right.  \label{equation1}%
\end{equation}
is well-posed and find explicitly the corresponding semigroup, which is a
Feller semigroup, see Theorem \ref{Theorem 3}.

Equations of type (\ref{equation1}) appeared as master equations in several
models that describe the dynamics of certain hierarchic complex systems, see
e.g. \cite{Av-4}-\cite{Av-5}, \cite{Ch-Z-1}, \cite{Kozyrev SV} and the
references therein. From a physical perspective, it is expected that all these
master equations should describe the evolution of a probability density, our
results show that this is, indeed, the case for a large class of symbols.

An interesting open problem consists in extending the results presented here
to the case of symbols of type $p\left(  x,t,\xi\right)  =\sum_{j=1}^{m}%
b_{j}(x,t)\boldsymbol{\psi}_{j}\left(  \xi\right)  $, where the
$\boldsymbol{\psi}_{j}:\mathbf{%
\mathbb{Q}
}_{p}^{n}\rightarrow%
\mathbb{R}
_{+}$ are radial, continuous, negative definite functions and the coefficients
$b_{j}:\mathbf{%
\mathbb{Q}
}_{p}^{n}\times%
\mathbb{R}
_{+}\rightarrow%
\mathbb{R}
_{+}$ are positive functions satisfying some suitable condition.

The article is organized as follows. In Section \ref{Fourier Analysis}, we
collect some basic results on the $p$-adic analysis and fix the notation that
we will use through the article. In Section \ref{positive and negative}, we
collect some known results on positive definite and negative definite
functions on $%
\mathbb{Q}
_{p}^{n}$ and show the existence of certain `exotic' negative definite
functions in non-Archimedean setting. In Section \ref{Function Spaces}, we
construct the spaces $B_{\boldsymbol{\psi},l}(\mathbf{%
\mathbb{Q}
}_{p}^{n},\mathbb{C})$,$\ B_{\boldsymbol{\psi},\infty}(\mathbf{%
\mathbb{Q}
}_{p}^{n},\mathbb{C})$. The space $B_{\boldsymbol{\psi},\infty}(\mathbb{R})$
is densely and continuously embedded in $C_{0}(%
\mathbb{Q}
_{p}^{n},\mathbb{R})$, see Lemma \ref{Lemma1}. In Section
\ref{Pseudodifferential operators and Feller semigroups}, we recall the
Yosida-Hille-Ray Theorem in the setting of $(\mathbf{%
\mathbb{Q}
}_{p}^{n},||\cdot||_{p})$. Moreover, we introduce a new class of
pseudodifferential operators attached to negative definite functions. We show
that these operators have a closed extensions which are the generators of a
Feller semigroups, see Theorem \ref{Feller semigroups}. In Section
\ref{Parabolic-Type Equations}, we study the Cauchy problem (\ref{equation1}),
see Theorem \ref{Theorem 3}.

\section{\label{Fourier Analysis}Fourier Analysis on $%
\mathbb{Q}
_{p}^{n}$: Essential Ideas}

\subsection{The field of $p$-adic numbers}

Along this article $p$ will denote a prime number. The field of $p-$adic
numbers $%
\mathbb{Q}
_{p}$ is defined as the completion of the field of rational numbers
$\mathbb{Q}$ with respect to the $p-$adic norm $|\cdot|_{p}$, which is defined
as
\[
\left\vert x\right\vert _{p}=\left\{
\begin{array}
[c]{lll}%
0 & \text{if} & x=0\\
&  & \\
p^{-\gamma} & \text{if} & x=p^{\gamma}\frac{a}{b}\text{,}%
\end{array}
\right.
\]
where $a$ and $b$ are integers coprime with $p$. The integer $\gamma:=ord(x)
$, with $ord(0):=+\infty$, is called the\textit{\ }$p-$\textit{adic order of}
$x$.

Any $p-$adic number $x\neq0$ has a unique expansion of the form
\[
x=p^{ord(x)}\sum_{j=0}^{\infty}x_{j}p^{j},
\]
where $x_{j}\in\{0,1,2,\dots,p-1\}$ and $x_{0}\neq0$. By using this expansion,
we define \textit{the fractional part of }$x\in\mathbb{Q}_{p}$, denoted
$\{x\}_{p}$, as the rational number
\[
\left\{  x\right\}  _{p}=\left\{
\begin{array}
[c]{lll}%
0 & \text{if} & x=0\text{ or }ord(x)\geq0\\
&  & \\
p^{ord(x)}\sum_{j=0}^{-ord_{p}(x)-1}x_{j}p^{j} & \text{if} & ord(x)<0.
\end{array}
\right.
\]
In addition, any non-zero $p-$adic number can be represented uniquely as
$x=p^{ord(x)}ac\left(  x\right)  $ where $ac\left(  x\right)  =\sum
_{j=0}^{\infty}x_{j}p^{j}$, $x_{0}\neq0$, is called the \textit{angular
component} of $x$. Notice that $\left\vert ac\left(  x\right)  \right\vert
_{p}=1$.

We extend the $p-$adic norm to $%
\mathbb{Q}
_{p}^{n}$ by taking
\[
||x||_{p}:=\max_{1\leq i\leq n}|x_{i}|_{p},\text{ for }x=(x_{1},\dots
,x_{n})\in%
\mathbb{Q}
_{p}^{n}.
\]
We define $ord(x)=\min_{1\leq i\leq n}\{ord(x_{i})\}$, then $||x||_{p}%
=p^{-ord(x)}$.\ The metric space $\left(
\mathbb{Q}
_{p}^{n},||\cdot||_{p}\right)  $ is a complete ultrametric space, which is a
totally disconnected topological space. For $r\in\mathbb{Z}$, denote by
$B_{r}^{n}(a)=\{x\in%
\mathbb{Q}
_{p}^{n};||x-a||_{p}\leq p^{r}\}$ \textit{the ball of radius }$p^{r}$
\textit{with center at} $a=(a_{1},\dots,a_{n})\in%
\mathbb{Q}
_{p}^{n}$, and take $B_{r}^{n}(0):=B_{r}^{n}$. Note that $B_{r}^{n}%
(a)=B_{r}(a_{1})\times\cdots\times B_{r}(a_{n})$, where $B_{r}(a_{i}):=\{x\in%
\mathbb{Q}
_{p};|x_{i}-a_{i}|_{p}\leq p^{r}\}$ is the one-dimensional ball of radius
$p^{r}$ with center at $a_{i}\in%
\mathbb{Q}
_{p}$. The ball $B_{0}^{n}$ equals the product of $n$ copies of $B_{0}%
=\mathbb{Z}_{p}$, \textit{the ring of }$p-$\textit{adic integers of }$%
\mathbb{Q}
_{p}$. We also denote by $S_{r}^{n}(a)=\{x\in\mathbb{Q}_{p}^{n};||x-a||_{p}%
=p^{r}\}$ \textit{the sphere of radius }$p^{r}$ \textit{with center at}
$a=(a_{1},\dots,a_{n})\in%
\mathbb{Q}
_{p}^{n}$, and take $S_{r}^{n}(0):=S_{r}^{n}$. We notice that $S_{0}%
^{1}=\mathbb{Z}_{p}^{\times}$ (the group of units of $\mathbb{Z}_{p}$), but
$\left(  \mathbb{Z}_{p}^{\times}\right)  ^{n}\subsetneq S_{0}^{n}$. The balls
and spheres are both open and closed subsets in $%
\mathbb{Q}
_{p}^{n}$. In addition, two balls in $%
\mathbb{Q}
_{p}^{n}$ are either disjoint or one is contained in the other.

As a topological space $\left(
\mathbb{Q}
_{p}^{n},||\cdot||_{p}\right)  $ is totally disconnected, i.e. the only
connected \ subsets of $%
\mathbb{Q}
_{p}^{n}$ are the empty set and the points. A subset of $%
\mathbb{Q}
_{p}^{n}$ is compact if and only if it is closed and bounded in $%
\mathbb{Q}
_{p}^{n}$, see e.g. \cite[Section 1.3]{V-V-Z}, or \cite[Section 1.8]{Alberio
et al}. The balls and spheres are compact subsets. Thus $\left(
\mathbb{Q}
_{p}^{n},||\cdot||_{p}\right)  $ is a locally compact topological space.

We will use $\Omega\left(  p^{-r}||x-a||_{p}\right)  $ to denote the
characteristic function of the ball $B_{r}^{n}(a)$. We will use the notation
$1_{A}$ for the characteristic function of a set $A$. Along the article
$d^{n}x$ will denote a Haar measure on $%
\mathbb{Q}
_{p}^{n}$ normalized so that $\int_{%
\mathbb{Z}
_{p}^{n}}d^{n}x=1.$

\subsection{Some function spaces}

A complex-valued function $\varphi$ defined on $%
\mathbb{Q}
_{p}^{n}$ is \textit{called locally constant} if for any $x\in%
\mathbb{Q}
_{p}^{n}$ there exist an integer $l(x)\in\mathbb{Z}$ such that
\[
\varphi(x+x^{\prime})=\varphi(x)\text{ for }x^{\prime}\in B_{l(x)}^{n}.
\]
A function $\varphi:%
\mathbb{Q}
_{p}^{n}\rightarrow\mathbb{C}$ is called a \textit{Bruhat-Schwartz function
(or a test function)} if it is locally constant with compact support. The
$\mathbb{C}$-vector space of Bruhat-Schwartz functions is denoted by
$\mathcal{D}:=\mathcal{D}(%
\mathbb{Q}
_{p}^{n})$. Let $\mathcal{D}^{\prime}:=\mathcal{D}^{\prime}(%
\mathbb{Q}
_{p}^{n})$ denote the set of all continuous functional (distributions) on
$\mathcal{D}$. The natural pairing $\mathcal{D}^{\prime}(%
\mathbb{Q}
_{p}^{n})\times\mathcal{D}(%
\mathbb{Q}
_{p}^{n})\rightarrow\mathbb{C}$ is denoted as $\left(  T,\varphi\right)  $ for
$T\in\mathcal{D}^{\prime}(%
\mathbb{Q}
_{p}^{n})$ and $\varphi\in\mathcal{D}(%
\mathbb{Q}
_{p}^{n})$, see e.g. \cite[Section 4.4]{Alberio et al}.

Every $f\in$ $L_{loc}^{1}(%
\mathbb{Q}
_{p}^{n})$ defines a distribution $f\in\mathcal{D}^{\prime}\left(
\mathbb{Q}
_{p}^{n}\right)  $ by the formula
\[
\left(  f,\varphi\right)  =%
{\textstyle\int\limits_{\mathbb{Q} _{p}^{n}}}
f\left(  x\right)  \varphi\left(  x\right)  d^{n}x.
\]
Such distributions are called \textit{regular distributions}.

We will denote by $\mathcal{D}_{\mathbb{R}}:=\mathcal{D}_{\mathbb{R}}(%
\mathbb{Q}
_{p}^{n})$, the $\mathbb{R}$-vector space of test functions, and by
$\mathcal{D}_{\mathbb{R}}^{\prime}:=\mathcal{D}_{\mathbb{R}}^{\prime}(%
\mathbb{Q}
_{p}^{n})$, the $\mathbb{R}$-vector space of distributions.

Given $\rho\in\lbrack0,\infty)$, we denote by $L^{\rho}:=L^{\rho}\left(
\mathbb{Q}
_{p}^{n}\right)  :=L^{\rho}\left(
\mathbb{Q}
_{p}^{n},d^{n}x\right)  ,$ the $%
\mathbb{C}
-$vector space of all the complex valued functions $g$ satisfying $\int_{%
\mathbb{Q}
_{p}^{n}}\left\vert g\left(  x\right)  \right\vert ^{\rho}d^{n}x<\infty$, and
$L^{\infty}\allowbreak:=L^{\infty}\left(
\mathbb{Q}
_{p}^{n}\right)  =L^{\infty}\left(
\mathbb{Q}
_{p}^{n},d^{n}x\right)  $ denotes the $%
\mathbb{C}
-$vector space of all the complex valued functions $g$ such that the essential
supremum of $|g|$ is bounded. The corresponding $\mathbb{R}$-vector spaces are
denoted as $L_{\mathbb{R}}^{\rho}\allowbreak:=L_{\mathbb{R}}^{\rho}\left(
\mathbb{Q}
_{p}^{n}\right)  =L_{\mathbb{R}}^{\rho}\left(
\mathbb{Q}
_{p}^{n},d^{n}x\right)  $, $1\leq\rho\leq\infty$.

Denote by $C(%
\mathbb{Q}
_{p}^{n})$ the $\mathbb{C}$-vector space of all complex-valued continuous
functions. Set
\[
C_{0}(%
\mathbb{Q}
_{p}^{n},\mathbb{C}):=\left\{  f:%
\mathbb{Q}
_{p}^{n}\rightarrow%
\mathbb{C}
;\text{ }f\text{ is continuous and }\lim_{||x||_{p}\rightarrow\infty
}f(x)=0\right\}  ,
\]
where $\lim_{||x||_{p}\rightarrow\infty}f(x)=0$ means that for every
$\epsilon>0$ there exists a compact subset $B(\epsilon)$ such that
$|f(x)|<\epsilon$ for $x\in%
\mathbb{Q}
_{p}^{n}\backslash B(\epsilon).$ We recall that $(C_{0}(%
\mathbb{Q}
_{p}^{n},\mathbb{C}),||\cdot||_{L^{\infty}})$ is a Banach space. The
corresponding $\mathbb{R}$-vector space will be denoted as $C_{0}(%
\mathbb{Q}
_{p}^{n},\mathbb{R})$.

\subsection{Fourier transform}

Set $\chi_{p}(y)=\exp(2\pi i\{y\}_{p})$ for $y\in%
\mathbb{Q}
_{p}$. The map $\chi_{p}(\cdot)$ is an additive character on $%
\mathbb{Q}
_{p}$, i.e. a continuous map from $\left(
\mathbb{Q}
_{p},+\right)  $ into $S$ (the unit circle considered as multiplicative group)
satisfying $\chi_{p}(x_{0}+x_{1})=\chi_{p}(x_{0})\chi_{p}(x_{1})$,
$x_{0},x_{1}\in%
\mathbb{Q}
_{p}$. The additive characters of $%
\mathbb{Q}
_{p}$ form an Abelian group which is isomorphic to $\left(
\mathbb{Q}
_{p},+\right)  $, the isomorphism is given by $\xi\rightarrow\chi_{p}(\xi x)$,
see e.g. \cite[Section 2.3]{Alberio et al}.

Given $x=(x_{1},\dots,x_{n}),$ $\xi=(\xi_{1},\dots,\xi_{n})\in%
\mathbb{Q}
_{p}^{n}$, we set $x\cdot\xi:=\sum_{j=1}^{n}x_{j}\xi_{j}$. If $f\in L^{1}$ its
Fourier transform is defined by
\[
(\mathcal{F}f)(\xi)=\int_{%
\mathbb{Q}
_{p}^{n}}\chi_{p}(\xi\cdot x)f(x)d^{n}x,\quad\text{for }\xi\in%
\mathbb{Q}
_{p}^{n}.
\]
We will also use the notation $\mathcal{F}_{x\rightarrow\xi}f$ and
$\widehat{f}$\ for the Fourier transform of $f$. The Fourier transform is a
linear isomorphism from $\mathcal{D}(%
\mathbb{Q}
_{p}^{n})$ onto itself satisfying
\begin{equation}
(\mathcal{F}(\mathcal{F}f))(\xi)=f(-\xi), \label{FF(f)}%
\end{equation}
for every $f\in\mathcal{D}(%
\mathbb{Q}
_{p}^{n}),$ see e.g. \cite[Section 4.8]{Alberio et al}. If $f\in L^{2},$ its
Fourier transform is defined as
\[
(\mathcal{F}f)(\xi)=\lim_{k\rightarrow\infty}\int_{||x||_{p}\leq p^{k}}%
\chi_{p}(\xi\cdot x)f(x)d^{n}x,\quad\text{for }\xi\in%
\mathbb{Q}
_{p}^{n},
\]
where the limit is taken in $L^{2}.$ We recall that the Fourier transform is
unitary on $L^{2},$ i.e. $||f||_{L^{2}}=||\mathcal{F}f||_{L^{2}}$ for $f\in
L^{2}$ and that (\ref{FF(f)}) is also valid in $L^{2}$, see e.g. \cite[Chapter
$III$, Section 2]{Taibleson}.

The Fourier transform $\mathcal{F}\left[  T\right]  $ of a distribution
$T\in\mathcal{D}^{\prime}\left(
\mathbb{Q}
_{p}^{n}\right)  $ is defined by%
\[
\left(  \mathcal{F}\left[  T\right]  ,\varphi\right)  =\left(  T,\mathcal{F}%
\left[  \varphi\right]  \right)  \text{ for all }\varphi\in\mathcal{D}(%
\mathbb{Q}
_{p}^{n})\text{.}%
\]
The Fourier transform $T\rightarrow\mathcal{F}\left[  T\right]  $ is a linear
isomorphism from $\mathcal{D}^{\prime}\left(
\mathbb{Q}
_{p}^{n}\right)  $\ onto itself. Furthermore, $T=\mathcal{F}\left[
\mathcal{F}\left[  T\right]  \left(  -\xi\right)  \right]  $. We also use the
notation $\mathcal{F}_{x\rightarrow\xi}T$ and $\widehat{T}$ for the Fourier
transform of $T.$

\section{\label{positive and negative}Positive Definite and Negative Definite
Functions on $%
\mathbb{Q}
_{p}^{n}$}

In this section, we collect some results about positive definite and negative
definite functions that we will use along the article, we refer the reader to
\cite{Berg-Gunnar} for further details.

We denote by $\mathbb{N}$, the set of nonnegative integers.

\begin{definition}
\label{neg and pos def} A function $\varphi:\mathbb{%
\mathbb{Q}
}_{p}^{n}\rightarrow%
\mathbb{C}
$ is called positive definite, if$\ $%
\[%
{\textstyle\sum\nolimits_{i=1}^{m}}
{\textstyle\sum\nolimits_{j=1}^{m}}
\varphi(x_{i}-x_{j})\lambda_{i}\overline{\lambda}_{j}\geq0
\]
for all $m\in\mathbb{N}\backslash\{0\},$ $x_{1},\ldots,x_{m}\in$ $\mathbb{%
\mathbb{Q}
}_{p}^{n}$ and $\lambda_{1},\ldots,\lambda_{m}$ $\in$ $\mathbb{C}$. Here,
$\overline{\lambda}_{j}$ denotes the complex conjugate of $\lambda_{j}.$
\end{definition}

The set of positive definite functions on $%
\mathbb{Q}
_{p}^{n}$ is denoted as $\mathcal{P}(%
\mathbb{Q}
_{p}^{n})$ and the subset of $\mathcal{P}(%
\mathbb{Q}
_{p}^{n})$ consisting of the continuous positive definite functions on $%
\mathbb{Q}
_{p}^{n}$ is denoted as $\mathcal{C}\mathcal{P}(%
\mathbb{Q}
_{p}^{n})$. The following assertions hold: (i) $\mathcal{P}(%
\mathbb{Q}
_{p}^{n})$ is a convex cone which is closed in the topology of pointwise
convergence on $%
\mathbb{Q}
_{p}^{n}$; (ii) if $\varphi_{1}$, $\varphi_{2}\in\mathcal{P}(%
\mathbb{Q}
_{p}^{n})$, then $\varphi_{1}\varphi_{2}\in\mathcal{P}(%
\mathbb{Q}
_{p}^{n})$; the non-negative constant functions belong to $\mathcal{P}(%
\mathbb{Q}
_{p}^{n})$; (iii) $\mathcal{CP}(%
\mathbb{Q}
_{p}^{n})$ is a convex cone which is a closed subset of the set of continuous
complex-valued functions in the topology of compact convergence cf.
\cite[Proposition 3.6]{Berg-Gunnar}.

\begin{example}
\label{example pos def}(i) We set $\mathbb{R}_{+}:=\{x\in\mathbb{R}:x\geq0\}$.
Let $J:$\textbf{\ }$%
\mathbb{Q}
_{p}^{n}\rightarrow\mathbb{R}_{+}$ be a radial (i.e. $J(x)=J(||x||_{p})$) and
continuous function. In addition, we assume that $\ \int_{%
\mathbb{Q}
_{p}^{n}}J(||x||_{p})d^{n}x=1$. By a direct calculation one verifies that
$\widehat{J}(\xi)$ is a radial, continuous and positive definite function on $%
\mathbb{Q}
_{p}^{n}$ and moreover $|\widehat{J}(||\xi||_{p})|\leq1.$

\noindent(ii) The additive character $x\rightarrow$ $\chi_{p}(x\cdot\alpha)$,
for $\alpha\in%
\mathbb{Q}
_{p}^{n}$, is a continuous, positive definite (complex-valued) function on $%
\mathbb{Q}
_{p}^{n}$, see e.g. \cite[p. 13]{Berg-Gunnar}.
\end{example}

\begin{definition}
A function $\psi:%
\mathbb{Q}
_{p}^{n}\rightarrow%
\mathbb{C}
$ is called negative definite, if
\begin{equation}%
{\textstyle\sum\nolimits_{i=1}^{m}}
{\textstyle\sum\nolimits_{j=1}^{m}}
\left(  \psi(x_{i})+\overline{\psi(x_{j})}-\psi(x_{i}-x_{j})\right)
\lambda_{i}\overline{\lambda}_{j}\geq0 \label{def negative definite}%
\end{equation}
for all $m\in\mathbb{N}\backslash\{0\}$ , $x_{1},\ldots,x_{m}$ $\in$ $%
\mathbb{Q}
_{p}^{n}$ and $\lambda_{1},\ldots,\lambda_{m}\in\mathbb{C}$.
\end{definition}

We denote by $\mathcal{N}(%
\mathbb{Q}
_{p}^{n})$ the set of negative definite functions on $%
\mathbb{Q}
_{p}^{n}$ and by $\mathcal{CN}(%
\mathbb{Q}
_{p}^{n})$ the set of continuous negative definite functions on $%
\mathbb{Q}
_{p}^{n}$. The following assertions hold: (i) $\mathcal{N}(%
\mathbb{Q}
_{p}^{n})$ is a convex cone which is closed in the topology of pointwise
convergence on $%
\mathbb{Q}
_{p}^{n}$; (ii) The non-negative constant functions belong to $\mathcal{N}(%
\mathbb{Q}
_{p}^{n})$; (iii) $\mathcal{CN}(%
\mathbb{Q}
_{p}^{n})$ is a convex cone which is closed in the topology of compact
convergence on $%
\mathbb{Q}
_{p}^{n}$, cf. \cite[Proposition 7.4]{Berg-Gunnar}.

Furthermore, if $\psi:%
\mathbb{Q}
_{p}^{n}\rightarrow%
\mathbb{R}
$ is negative definite function, then $\psi(-x)=\psi(x)$ and $\psi(x)\geq
\psi(0)\geq0$ for all $x\in%
\mathbb{Q}
_{p}^{n}$, see e.g. \cite[Proposition 7.5]{Berg-Gunnar}.

\begin{example}
\label{Example2}Let $J:$\textbf{\ }$%
\mathbb{Q}
_{p}^{n}\rightarrow\mathbb{R}_{+}$ the function given in Example
\ref{example pos def}-(i). By using Corollary 7.7 in \cite[Theorem
7.8]{Berg-Gunnar}, the function $\widehat{J}(0)-\widehat{J}(||\xi
||_{p})=1-\widehat{J}(||\xi||_{p})$ is negative definite. On the other hand,
we have that $0\leq1-\widehat{J}(||\xi||_{p})\leq2,$ $\xi\in%
\mathbb{Q}
_{p}^{n},$ see e.g. \cite[Lemma 1-(i)]{To-Z}.
\end{example}

\begin{example}
\label{example neg def}In \cite{R-Zu}, see also \cite{Koch},
Rodr\'{\i}guez-Vega and Z\'{u}\~{n}iga-Galindo considered the following Cauchy
problem:
\begin{equation}
\left\{
\begin{array}
[c]{ll}%
\frac{\partial u}{\partial t}(x,t)+a(D_{T}^{\beta}u)(x,t)=f(x,t)\text{,} &
t\in(0,T_{0}]\text{,\ }x\in%
\mathbb{Q}
_{p}^{n}\\
& \\
u(x,0)=\varphi(x)\text{,} &
\end{array}
\right.  \label{Caucht_problem}%
\end{equation}
where $a$, $\beta$, $T_{0}$ are positive real numbers, and $(D_{T}^{\beta
}h)(x)=%
\mathcal{F}%
_{\xi\rightarrow x}^{-1}(||\xi||_{p}^{\beta}%
\mathcal{F}%
_{x\rightarrow\xi}h)$ is the Taibleson operator. They established that
$e^{-at||\xi||_{p}^{\beta}}\in L^{1}(%
\mathbb{Q}
_{p}^{n})$ for $t>0,$ and that $Z(x,t)=%
\mathcal{F}%
_{\xi\rightarrow x}^{-1}(e^{-at||\xi||_{p}^{\beta}})$, for $a$, $t>0$, is a
transition function of a Markov process with space state $%
\mathbb{Q}
_{p}^{n}$, cf. \cite[Proposition 1 and Theorem 2]{R-Zu}. By using a theorem
due to Bochner, see \cite[Theorem 3.12]{Berg-Gunnar}, the function
$e^{-at||\xi||_{p}^{\beta}}$, for $t>0$, is positive definite, and by a
theorem due to Schoenberg, see \cite[Theorem 7.8]{Berg-Gunnar}, $a||\xi
||_{p}^{\beta}$ is a negative definite function, for any $\beta>0$.
\end{example}

\begin{example}
\label{Example a(x,)}Take $m\in%
\mathbb{N}
\backslash\{0\}$ and let $b_{j},$ $j=1,\ldots,m,$ be positive real numbers.
Let $0<\alpha_{1}\leq\ldots\leq\alpha_{m}$ be positive constants. Consider the
function
\begin{equation}
q(\xi)=%
{\displaystyle\sum\nolimits_{j=1}^{m}}
b_{j}||\xi||_{p}^{\alpha_{j}},\label{a(x,.) example}%
\end{equation}
for $\xi\in%
\mathbb{Q}
_{p}^{n}$. By Example \ref{example neg def}, and the fact that $\mathcal{N}(%
\mathbb{Q}
_{p}^{n})$ is a convex cone, $q(\xi)$ is radial, continuous, negative definite function.
\end{example}

\begin{remark}
\label{others examples def neg}(i) It is relevant to mention that the type of
functions given in (\ref{a(x,.) example}), with some $\alpha_{j}>2$, occurs
only in the non-Archimedean setting, since any function $\psi:%
\mathbb{R}
^{n}\rightarrow%
\mathbb{R}
$ locally bounded and negative definite, satisfies
\[
\left\vert \psi(\xi)\right\vert \leq C_{\psi}(1+||\xi||_{%
\mathbb{R}
}^{2}),
\]
for some $C_{\psi}>0$ and for all $\xi\in%
\mathbb{R}
^{n}$, see e.g. \cite[Lemma 3.6.22]{Jacob-vol-1}.

\noindent(ii) Let $h\left(  y\right)  =\sum_{j=0}^{\infty}a_{j}y^{j}$,
$a_{j}\in\mathbb{R}_{+}$, be a convergent series in $\mathbb{R}_{+}$, which
defines a non-constant function. By using Example \ref{example neg def} and
the fact that $\mathcal{N}(%
\mathbb{Q}
_{p}^{n})$ is closed in the pointwise topology, it follows that $h\left(
||x||_{p}\right)  =\sum_{j=0}^{\infty}a_{j}||x||_{p}^{j}$ is a negative
definite function.
\end{remark}

\begin{lemma}
\label{exp Fi negative definite}(i) Let $\psi:%
\mathbb{Q}
_{p}^{n}\rightarrow%
\mathbb{C}
$ be a negative definite function such that $\xi\rightarrow\left[  \psi\left(
\xi\right)  \right]  ^{j}$ is also a negative definite function for any $j\in%
\mathbb{N}
$. Then $e^{\psi(\xi)}$ is a negative definite function. (ii) Set $\psi
_{0}\left(  \xi\right)  :=\sum_{j=1}^{\infty}c_{j}||\xi||_{p}^{\alpha_{j}}$
with $c_{j}\geq0$, $\alpha_{j}\in\mathbb{N}$ such that the real series
$\sum_{j=1}^{\infty}c_{j}y^{\alpha_{j}}$ defines a non-constant real function.
Then for any $j\in%
\mathbb{N}
\backslash\{0\}$,
\begin{equation}
e^{e^{.^{.^{.e^{\psi_{0}\left(  \xi\right)  }}}}}\text{, }j-\text{powers }
\label{expexpformula}%
\end{equation}
is a continuous and negative definite function on $%
\mathbb{Q}
_{p}^{n}$.
\end{lemma}

\begin{proof}
By the hypothesis, $\psi_{m}:=%
{\displaystyle\sum\nolimits_{j=0}^{m}}
\frac{1}{j!}\left[  \psi(\xi)\right]  ^{j}$, $m\in%
\mathbb{N}
$, is negative definite, and since $\mathcal{N}(%
\mathbb{Q}
_{p}^{n})$ is closed in the pointwise topology, we have that $e^{\psi(\xi)}$
is negative definite. By using Remark \ref{others examples def neg}-(ii),
$\psi_{0}\left(  \xi\right)  =\sum_{j=1}^{\infty}c_{j}||\xi||_{p}^{\alpha_{j}%
}$ is negative definite, and $\left[  \psi_{0}(\xi)\right]  ^{k}=\left(
\sum_{j=1}^{\infty}c_{j}||\xi||_{p}^{\alpha_{j}}\right)  ^{k}=\sum
_{j=1}^{\infty}d_{j}||\xi||_{p}^{\beta_{j}}$, with $d_{j}=d_{j}\left(
k\right)  \geq0$, $\beta_{j}=\beta_{j}\left(  k\right)  \in%
\mathbb{N}
$, is also negative definite function. By the first part $e^{\psi_{0}\left(
\xi\right)  }$ is negative definite. By induction on $j$ we obtain
(\ref{expexpformula}).
\end{proof}

Negative definite functions of form (\ref{expexpformula}) can only occur in
the non-Archimedean setting.

From now on, $\boldsymbol{\psi}:\mathbf{%
\mathbb{Q}
}_{p}^{n}\rightarrow%
\mathbb{C}
$ (or $%
\mathbb{R}
$) denotes a radial, continuous and negative definite function. We consider
three subclasses of negative definite functions, however, we do not expect
that this classification be complete.

\begin{definition}
$\boldsymbol{\psi}:\mathbf{%
\mathbb{Q}
}_{p}^{n}\rightarrow$ $%
\mathbb{C}
$ (or $%
\mathbb{R}
$) is called of type 0, if there exists a positive constant
$C:=C(\boldsymbol{\psi})$ such that
\[
|\boldsymbol{\psi}(||\xi||_{p})|\leq C,\text{ for all }\xi\in%
\mathbb{Q}
_{p}^{n}.
\]

\end{definition}

\begin{definition}
$\boldsymbol{\psi}:\mathbf{%
\mathbb{Q}
}_{p}^{n}\rightarrow$ $%
\mathbb{C}
$ (or $%
\mathbb{R}
$) is called of type 1, if\textbf{\ }there exist positive constants
$C_{0}(\boldsymbol{\psi}):=C_{0},$ $C_{1}(\boldsymbol{\psi}):=C_{1}$,
$\beta_{0}\left(  \boldsymbol{\psi}\right)  :=\beta_{0}\in%
\mathbb{R}
_{+}\backslash\{0\}$ and $\beta_{1}\left(  \boldsymbol{\psi}\right)
:=\beta_{1}\in%
\mathbb{R}
_{+}\backslash\{0\},$ with $\beta_{1}\geq\beta_{0},$ such that
\[
C_{0}\left[  \max\left\{  1,||\xi||_{p}\right\}  \right]  ^{\beta_{0}}\leq
\max\{1,|\boldsymbol{\psi}(||\xi||_{p})|\}\leq C_{1}\left[  \max\left\{
1,||\xi||_{p}\right\}  \right]  ^{\beta_{1}},
\]
for all $\xi\in\mathbf{%
\mathbb{Q}
}_{p}^{n}$.
\end{definition}

\begin{definition}
$\boldsymbol{\psi}:\mathbf{%
\mathbb{Q}
}_{p}^{n}\rightarrow$ $%
\mathbb{C}
$ (or $%
\mathbb{R}
$) is called of type 2, if\textbf{\ }for all $\beta\geq1,$\ there is a
positive constant $C:=C(\boldsymbol{\psi,}\beta)$ such that
\[
\max\left\{  1,\left\vert \boldsymbol{\psi}(||\xi||_{p})\right\vert \right\}
>C\left[  \max\left\{  1,||\xi||_{p}\right\}  \right]  ^{\beta}\text{, for all
}\xi\in\mathbf{%
\mathbb{Q}
}_{p}^{n}.
\]

\end{definition}

\section{\label{Function Spaces}Function Spaces Related to Negative Definite
Functions}

Along this section $\boldsymbol{\psi}:%
\mathbb{Q}
_{p}^{n}\rightarrow\mathbb{C}$ denotes a negative definite, radial and
continuous function of type $1$ or $2,$ unless otherwise stated. In addition,
we will assume that%
\begin{equation}
0<\sup_{\xi\in\mathbb{Z}_{p}^{n}}|\boldsymbol{\psi}(||\xi||_{p})|\leq1\text{.}
\label{Condition_Psi}%
\end{equation}
This condition is achieved by multiplying $\boldsymbol{\psi}$ by a suitable
positive constant. Condition (\ref{Condition_Psi}) implies that%
\begin{equation}
\varphi_{l}\left(  x\right)  :=[\max\left\{  1,\left\vert \boldsymbol{\psi
}(||\xi||_{p})\right\vert \right\}  ]^{l}\text{, }l\in\mathbb{N}\text{, is a
locally constant function,} \label{locally_const_condition}%
\end{equation}
more precisely, $\varphi_{l}\left(  x+x^{\prime}\right)  =\varphi_{l}\left(
x\right)  $ for any $x^{\prime}\in\mathbb{Z}_{p}^{n}$.

In this section, we introduce two classes of function spaces related to
$\boldsymbol{\psi}$, namely $B_{\boldsymbol{\psi},l}\left(  \mathbb{C}\right)
$, $l\in%
\mathbb{N}
,$ and $B_{\boldsymbol{\psi},\infty}\left(  \mathbb{C}\right)  $. These spaces
are generalizations of the spaces $H_{%
\mathbb{C}
}(l),$ $l\in%
\mathbb{N}
,$ and $H_{%
\mathbb{C}
}(\infty)$ introduced by Z\'{u}\~{n}iga-Galindo in \cite{Zu1}, see also
\cite{Zu0}. The results presented in this section can be established by using
the techniques presented in \cite{Zu1}, \cite{Zu0}.

For $\varphi,\gamma\in\mathcal{D}(%
\mathbb{Q}
_{p}^{n})$, and $l\in%
\mathbb{N}
$, we define the following scalar product:
\[
\left\langle \varphi,\gamma\right\rangle _{\boldsymbol{\psi},l}=\int_{\mathbf{%
\mathbb{Q}
}_{p}^{n}}[\max\left\{  1,\left\vert \boldsymbol{\psi}(||\xi||_{p})\right\vert
\right\}  ]^{l}\widehat{\varphi}(\xi)\overline{\widehat{\gamma}(\xi)}d^{n}%
\xi,\text{ }%
\]
where the bar denotes the complex conjugate. We also set
\[
||\varphi||_{\boldsymbol{\psi},l}^{2}:=\left\langle \varphi,\varphi
\right\rangle _{\boldsymbol{\psi},l}.
\]
Notice that $||\cdot||_{\boldsymbol{\psi},l}\leq||\cdot||_{\boldsymbol{\psi
},m}$ for $l\leq m.$ Let us denote by $B_{\boldsymbol{\psi},l}\left(
\mathbb{C}\right)  :=B_{\boldsymbol{\psi},l}(\mathbf{%
\mathbb{Q}
}_{p}^{n},\mathbb{C})$ the completion of $\mathcal{D}(%
\mathbb{Q}
_{p}^{n})$ with respect to $\left\langle \cdot,\cdot\right\rangle
_{\boldsymbol{\psi},l}.$ Then $B_{\boldsymbol{\psi},m}\left(  \mathbb{C}%
\right)  \hookrightarrow B_{\boldsymbol{\psi},l}\left(  \mathbb{C}\right)  $
(continuous embedding) for $l\leq m.$

We set
\[
B_{\boldsymbol{\psi},\infty}\left(  \mathbb{C}\right)  :=B_{\boldsymbol{\psi
},\infty}(\mathbf{%
\mathbb{Q}
}_{p}^{n},\mathbb{C})=\cap_{l\in%
\mathbb{N}
}B_{\boldsymbol{\psi},l}\left(  \mathbb{C}\right)  .
\]
Notice that $B_{\boldsymbol{\psi},0}\left(  \mathbb{C}\right)  =L^{2}(%
\mathbb{Q}
_{p}^{n})$ and $\mathcal{D}(%
\mathbb{Q}
_{p}^{n})\subset B_{\boldsymbol{\psi},\infty}\left(  \mathbb{C}\right)
\subset$ $L^{2}(%
\mathbb{Q}
_{p}^{n}).$ With the topology induced by the family of seminorms
$\{||\cdot||_{\boldsymbol{\psi},l}\}_{l\in%
\mathbb{N}
},$ $B_{\boldsymbol{\psi},\infty}\left(  \mathbb{C}\right)  $ becomes a
locally convex topological space, which is metrizable. Indeed,
\[
d_{\boldsymbol{\psi}}(f,g):=\max_{l\in%
\mathbb{N}
}\left\{  2^{-l}\frac{||f-g||_{\boldsymbol{\psi},l}}%
{1+||f-g||_{\boldsymbol{\psi},l}}\right\}  ,\text{ with }f,g\in
B_{\boldsymbol{\psi},\infty},
\]
is a metric for the topology of $B_{\boldsymbol{\psi},\infty}\left(
\mathbb{C}\right)  $ considered as a locally convex topological space. A
sequence $\{f_{l}\}_{l\in%
\mathbb{N}
}$ in $(B_{\boldsymbol{\psi},\infty}\left(  \mathbb{C}\right)
,d_{\boldsymbol{\psi}})$ converges to $f\in B_{\boldsymbol{\psi},\infty
}\left(  \mathbb{C}\right)  $ if and only if, $\{f_{l}\}_{l\in%
\mathbb{N}
}$ converges to $f$ in the norm $||\cdot||_{\boldsymbol{\psi,}l}$ for all
$l\in%
\mathbb{N}
.$ From this observation, it follows that the topology on $B_{\boldsymbol{\psi
},\infty}\left(  \mathbb{C}\right)  $ coincides with the projective limit
topology $\tau_{P}.$ An open neighborhood base at zero of $\tau_{P}$ is given
by the choice of $\epsilon>0$ and $l\in%
\mathbb{N}
,$ and the set
\[
U_{\epsilon,l}:=\{f\in B_{\boldsymbol{\psi},\infty};||f||_{\boldsymbol{\psi
,}l}<\epsilon\}.
\]
The space $B_{\boldsymbol{\psi},\infty}\left(  \mathbb{C}\right)  $ endowed
with the topology $\tau_{P}$ is a countable Hilbert space in the sense of
Gel'fand and Vilenkin, see e.g. \cite[Chapter I, Section 3.1]{Gelfand} or
\cite[Section 1.2]{Obata}. Furthermore $(B_{\boldsymbol{\psi},\infty}\left(
\mathbb{C}\right)  ,\tau_{P})$ is metrizable and complete and hence a
Fr\'{e}chet space cf. \cite[Lemma 3.3]{Zu1}. If $\boldsymbol{\psi=||\cdot
||}_{p}$, then $B_{\boldsymbol{\psi},l}\left(  \mathbb{C}\right)  $ coincides
with the space $H_{%
\mathbb{C}
}(l)$, respectively, $B_{\boldsymbol{\psi},\infty}\left(  \mathbb{C}\right)  $
coincides with the space $H_{%
\mathbb{C}
}(\infty)$, where $H_{%
\mathbb{C}
}(l)$ and $H_{%
\mathbb{C}
}(\infty)$ are the spaces introduced in \cite{Zu1}, see also \cite{Zu0}.

\begin{remark}
We will denote by $B_{\boldsymbol{\psi},l}(\mathbb{R}):=B_{\boldsymbol{\psi
},l}(%
\mathbb{Q}
_{p}^{n},\mathbb{R})$, for all $l\in\mathbb{N}$, and by $B_{\boldsymbol{\psi
},\infty}(\mathbb{R}):=B_{\boldsymbol{\psi},\infty}(%
\mathbb{Q}
_{p}^{n},\mathbb{R})\allowbreak$ the $\mathbb{R}$-vector spaces constructed
from $\mathcal{D}_{\mathbb{R}}(%
\mathbb{Q}
_{p}^{n})$. It is clear that $B_{\boldsymbol{\psi},l}(\mathbb{R}%
)\hookrightarrow B_{\boldsymbol{\psi},l}(\mathbb{C}),$ $l\in\mathbb{N},$ and
that $B_{\boldsymbol{\psi},\infty}(\mathbb{R})\hookrightarrow
B_{\boldsymbol{\psi},\infty}(\mathbb{C})$, where `$\hookrightarrow$' means
continuous embedding.
\end{remark}

\begin{lemma}
\label{Lemma1}The following assertions hold:

\noindent(i) the completion of the metric space $(\mathcal{D}(%
\mathbb{Q}
_{p}^{n}),d_{\boldsymbol{\psi}})$ is $(B_{\boldsymbol{\psi},\infty}\left(
\mathbb{C}\right)  ,d_{\boldsymbol{\psi}})$, which is a nuclear countably
Hilbert space;

\noindent(ii) $B_{\boldsymbol{\psi},l}(\mathbb{C})=\{f\in L^{2}%
;||f||_{\boldsymbol{\psi,}l}<\infty\}=\{T\in\mathcal{D}^{\prime}\left(
\mathbb{Q}
_{p}^{n}\right)  ;||T||_{\boldsymbol{\psi,}l}<\infty\}$;

\noindent(iii) $B_{\boldsymbol{\psi},\infty}(\mathbb{C})=\{f\in L^{2}%
;||f||_{\boldsymbol{\psi,}l}<\infty$ for every $l\in%
\mathbb{N}
\}$;

\noindent(iv) $B_{\boldsymbol{\psi},\infty}(\mathbb{C})=\{T\in\mathcal{D}%
^{\prime}\left(
\mathbb{Q}
_{p}^{n}\right)  ;||T||_{\boldsymbol{\psi,}l}<\infty$ for every $l\in%
\mathbb{N}
\}$;

\noindent(v) $B_{\boldsymbol{\psi},\infty}(\mathbb{C})$ is densely and
continuously embedded in $C_{0}(%
\mathbb{Q}
_{p}^{n},\mathbb{C})$;

\noindent(vi) $B_{\boldsymbol{\psi},\infty}(\mathbb{R})$ is densely and
continuously embedded in $C_{0}(%
\mathbb{Q}
_{p}^{n},\mathbb{R})$;

\noindent(vii) $B_{\boldsymbol{\psi},\infty}(\mathbb{C})\subset L^{1}$. In
particular, $\widehat{f}\in C_{0}(%
\mathbb{Q}
_{p}^{n},\mathbb{C})$ for $f\in B_{\boldsymbol{\psi},\infty}(\mathbb{C})$.
\end{lemma}

\begin{remark}
\label{Nota_Lemma1}The condition $T\in\mathcal{D}^{\prime}\left(
\mathbb{Q}
_{p}^{n}\right)  $, $||T||_{\boldsymbol{\psi,}l}<\infty$ assumes implicitly
that $\widehat{T}$ is a regular distribution. The equalities (ii)-(iv) in
Lemma \ref{Lemma1} are in the sense of vector spaces. The statements (i)-(iv)
are valid for the spaces $\mathcal{D}_{\mathbb{R}}(%
\mathbb{Q}
_{p}^{n})$, $B_{\boldsymbol{\psi},l}(\mathbb{R})$ and $B_{\boldsymbol{\psi
},\infty}(\mathbb{R})$.
\end{remark}

\begin{proof}
(i) The proof is similar to \cite[Lemma 3.4 and Theorem 3.6]{Zu1}.

(ii) Take $f\in B_{\boldsymbol{\psi},l}(%
\mathbb{Q}
_{p}^{n})$, then there exists a sequence $\{f_{n}\}_{n\in%
\mathbb{N}
}$ in $\mathcal{D}\left(
\mathbb{Q}
_{p}^{n}\right)  $ such that $f_{n}\overset{||.||_{\boldsymbol{\psi},l}%
}{\rightarrow}f$ for any $l\in\mathbb{N}$, i.e.%
\[
\lbrack\max\left\{  1,|\boldsymbol{\psi}(||\xi||_{p})|\right\}  ]^{\frac{l}%
{2}}\widehat{f_{n}}\overset{||.||_{L^{2}}}{\rightarrow}[\max\left\{
1,|\boldsymbol{\psi}(||\xi||_{p})|\right\}  ]^{\frac{l}{2}}\widehat{f}.
\]
By taking $l=0$ and using that $L^{2}$ is complete, we have $\widehat{f}\in
L^{2}$, i.e. $f\in L^{2}$. Conversely, take $f\in L^{2}$ such that
$[\max\left\{  1,|\boldsymbol{\psi}(||\xi||_{p})|\right\}  ]^{\frac{l}{2}%
}\widehat{f}\in L^{2}$. By using the fact that $\mathcal{D}\left(
\mathbb{Q}
_{p}^{n}\right)  $ is dense in $L^{2},$ there exists a sequence $\{f_{m}%
\}_{m\in%
\mathbb{N}
}$ in $\mathcal{D}\left(
\mathbb{Q}
_{p}^{n}\right)  $ such that $f_{m}$ $\underrightarrow{||.||_{L^{2}}}$
$[\max\left\{  1,|\boldsymbol{\psi}(||\xi||_{p})|\right\}  ]^{\frac{l}{2}%
}\widehat{f}$. We now define $g_{m}\left(  \xi\right)  :=\frac{f_{m}\left(
-\xi\right)  }{[\max\left\{  1,|\boldsymbol{\psi}(||\xi||_{p})|\right\}
]^{\frac{l}{2}}}\in\mathcal{D}\left(
\mathbb{Q}
_{p}^{n}\right)  $, see (\ref{locally_const_condition}). Then $\widehat{g}_{m}
$ $\underrightarrow{||.||_{\boldsymbol{\psi,}l}}$ $f,$ for any $l\in
\mathbb{N}$, i.e. $f\in B_{\boldsymbol{\psi},l}(%
\mathbb{Q}
_{p}^{n})$. The other equality follows from the fact that $T\in L^{2}$,
$||T||_{\boldsymbol{\psi,}l}<\infty\Leftrightarrow T\in\mathcal{D}^{\prime
}\left(
\mathbb{Q}
_{p}^{n}\right)  $, $||T||_{\boldsymbol{\psi,}l}<\infty$.

(iii) and (iv) are an immediate consequence of (ii).

(v) By using the fact that $\frac{1}{[\max\left\{  1,||\xi||_{p}\right\}
]^{r}}\in L^{1}$ for $r>n$, one verifies that
\begin{equation}
\frac{1}{[\max\left\{  1,|\boldsymbol{\psi}(||\xi||_{p})|\right\}  ]^{l}}\in
L^{1}\text{ if}\left\{
\begin{array}
[c]{l}%
\boldsymbol{\psi}\text{ is of type 1 and }l>\frac{n}{\beta_{0}}\\
\\
\boldsymbol{\psi}\text{ is of type 2 and }l\geq1.
\end{array}
\right.  \label{Integrability_condition}%
\end{equation}
Take $f\in B_{\boldsymbol{\psi},\infty}(\mathbb{C})$ with $\boldsymbol{\psi}$
of type 1. By using the Cauchy-Schwarz inequality and
(\ref{Integrability_condition}), we have for all $l>\frac{n}{\beta_{0}}$,%
\begin{align*}
\int_{\mathbf{%
\mathbb{Q}
}_{p}^{n}}\left\vert \widehat{f}(\xi)\right\vert d^{n}\xi &  =\int_{\mathbf{%
\mathbb{Q}
}_{p}^{n}}\frac{1}{[\max\left\{  1,|\boldsymbol{\psi}(||\xi||_{p})|\right\}
]^{\frac{l}{2}}}[\max\left\{  1,|\boldsymbol{\psi}(||\xi||_{p})|\right\}
]^{\frac{l}{2}}\left\vert \widehat{f}(\xi)\right\vert d^{n}\xi\\
&  \leq C||f||_{\boldsymbol{\psi},l}.
\end{align*}
Thus $\widehat{f}\in L^{1}$, and since the Fourier transform of a function in
$L^{1}$ is uniformly continuous, now by the Riemann-Lebesgue theorem, we
obtain that $f\in C_{0}(%
\mathbb{Q}
_{p}^{n},\mathbb{C})$, i.e. $B_{\boldsymbol{\psi},\infty}\subset C_{0}(%
\mathbb{Q}
_{p}^{n},\mathbb{C})$. In addition,
\begin{equation}
||f||_{L^{\infty}}\leq||\widehat{f}||_{L^{1}}\leq C||f||_{\boldsymbol{\psi}%
,l}\text{ for }l>\frac{n}{\beta_{0}}\text{.} \label{Key_inequality}%
\end{equation}
Since the topology of $B_{\boldsymbol{\psi},\infty}$ comes from the metric
$d_{\boldsymbol{\psi}}$, the continuity of the embedding $B_{\boldsymbol{\psi
},\infty}\rightarrow C_{0}(%
\mathbb{Q}
_{p}^{n},\mathbb{C})$ follows from (\ref{Key_inequality}), by using a standard
argument based in sequences.

The proof is similar for functions $\boldsymbol{\psi}$ of type 2.

(vi) By (i) and Remark \ref{Nota_Lemma1}, and (v), we have%
\[
\mathcal{D}_{\mathbb{R}}\left(
\mathbb{Q}
_{p}^{n}\right)  \hookrightarrow B_{\boldsymbol{\psi},\infty}(\mathbb{R}%
)\hookrightarrow B_{\boldsymbol{\psi},\infty}(\mathbb{C})\hookrightarrow
C_{0}(%
\mathbb{Q}
_{p}^{n},\mathbb{C})\text{,}%
\]
which implies that $B_{\boldsymbol{\psi},\infty}(\mathbb{R})\hookrightarrow
C_{0}(%
\mathbb{Q}
_{p}^{n},\mathbb{R})$. Finally, we recall that $\mathcal{D}_{\mathbb{R}%
}\left(
\mathbb{Q}
_{p}^{n}\right)  $ is dense in $C_{0}(%
\mathbb{Q}
_{p}^{n},\mathbb{R})$.

(vii) The proof of the fact $B_{\boldsymbol{\psi},\infty}(\mathbb{C})\subset
L^{1}$ uses the same argument given \cite{Zu1} for Theorem 3.15-(ii). Then, by
the Riemann-Lebesgue theorem, $\widehat{f}\in C_{0}(%
\mathbb{Q}
_{p}^{n},\mathbb{C})$ for $f\in B_{\boldsymbol{\psi},\infty}(\mathbb{C})$.
\end{proof}

Recall that $H_{\mathbb{C}}(l)=B_{\boldsymbol{\psi},l}(\mathbb{C})$ and
$H_{\mathbb{C}}(\infty)=B_{\boldsymbol{\psi},\infty}(\mathbb{C})$ when
$\boldsymbol{\psi}\mathbb{=}||\cdot||_{p}$, see \cite{Zu1}, \cite{Zu0}.

\begin{lemma}
$B_{\boldsymbol{\psi},\infty}(\mathbb{C})\hookrightarrow H_{\mathbb{C}}%
(\infty).$
\end{lemma}

\begin{proof}
By definition of the classes type 1 or 2,
\begin{equation}
||\varphi||_{\boldsymbol{\psi,}l}\geq C(\boldsymbol{\psi},l)||\varphi
||_{l}\text{ for }\varphi\in\mathcal{D}\left(
\mathbb{Q}
_{p}^{n}\right)  \text{ and}\ l\in%
\mathbb{N}
. \label{des in lemma}%
\end{equation}

By using the density of $\mathcal{D}\left(
\mathbb{Q}
_{p}^{n}\right)  $ in $(B_{\boldsymbol{\psi},l}(\mathbb{C}),||\cdot
||_{\boldsymbol{\psi},l})$, we conclude that $B_{\boldsymbol{\psi}%
,l}(\mathbb{C})\subseteq H_{\mathbb{C}}(l)$ for any $l\in%
\mathbb{N}
.$ Consequently $B_{\boldsymbol{\psi},\infty}(\mathbb{C})\subseteq
H_{\mathbb{C}}(\infty)$. To check the continuity of the identity map, we use
that if $f_{n}\overset{d_{\boldsymbol{\psi}}}{\rightarrow}f,$ i.e.
$f_{n}\overset{||\cdot||_{\boldsymbol{\psi,}l}}{\rightarrow}f$, for all $l\in%
\mathbb{N}
,$ then by (\ref{des in lemma}) $f_{n}\overset{||\cdot||_{l}}{\rightarrow}f$
for any $l\in%
\mathbb{N}
,$ which means that the sequence $\left\{  f_{n}\right\}  _{n\in\mathbb{N}}$
converges to $f$ in $H_{\mathbb{C}}(\infty)$.
\end{proof}

\section{\label{Pseudodifferential operators and Feller semigroups}%
Pseudodifferential operators and Feller semigroups}

\subsection{Yosida-Hille-Ray Theorem}

We recall the Yosida-Hille-Ray Theorem in the setting of $(\mathbf{%
\mathbb{Q}
}_{p}^{n},||\cdot||_{p})$. For a general discussion the reader may consult
\cite[Chapter 4]{E-K}.

A semigroup $\{T(t)\}_{t\geq0}$ on $C_{0}(\mathbf{%
\mathbb{Q}
}_{p}^{n},\mathbb{R})$ is said to be \textit{positive} if $\{T(t)\}_{t\geq0} $
is a positive operator for each $t\geq0$, i.e. it maps non-negative functions
to non-negative functions. An operator $(A,Dom(A))$ on $C_{0}(\mathbf{%
\mathbb{Q}
}_{p}^{n},\mathbb{R})$ is said to satisfy the \textit{positive maximum
principle} if whenever $f\in Dom(A)\subseteq C_{0}(\mathbf{%
\mathbb{Q}
}_{p}^{n},\mathbb{R})$, $x_{0}\in\mathbf{%
\mathbb{Q}
}_{p}^{n}$, and $\sup_{x\in\mathbf{%
\mathbb{Q}
}_{p}^{n}}f(x)=f(x_{0})\geq0$ we have $Af(x_{0})\leq0$.

We recall that every linear operator on $C_{0}(\mathbf{%
\mathbb{Q}
}_{p}^{n},\mathbb{R})$ satisfying the positive maximum principle is
dissipative, see e.g. \cite[Chapter 4, Lemma 2.1]{E-K}.

\begin{theorem}
[Hille-Yosida-Ray Theorem]\label{Hille-Yosida-Ray}\cite[Chapter 4, Theorem
2.2]{E-K} Assume that $(A,Dom(A))$ is a linear operator on $C_{0}(\mathbf{%
\mathbb{Q}
}_{p}^{n},\mathbb{R})$. The closure $\overline{A}$ of $A$ on $C_{0}(\mathbf{%
\mathbb{Q}
}_{p}^{n},\mathbb{R})$ is single-valued and generates a strongly continuous,
positive contraction semigroup $\{T(t)\}_{t\geq0}$ on $C_{0}(\mathbf{%
\mathbb{Q}
}_{p}^{n},\mathbb{R})$ if and only if:

\noindent(i) $Dom(A)$ is dense in $C_{0}(\mathbf{%
\mathbb{Q}
}_{p}^{n},\mathbb{R})$;

\noindent(ii) $A$ satisfies the positive maximum principle;

\noindent(iii) Rank$(\lambda I-A)$ is dense in $C_{0}(\mathbf{%
\mathbb{Q}
}_{p}^{n},\mathbb{R})$ for some $\lambda>0$.
\end{theorem}

\begin{definition}
A family of bounded linear operators $T_{t}:C_{0}(\mathbf{%
\mathbb{Q}
}_{p}^{n},\mathbb{R})\rightarrow C_{0}(\mathbf{%
\mathbb{Q}
}_{p}^{n},\mathbb{R})$ is called a Feller semigroup if

\noindent(i) $T_{s+t}=T_{s}T_{t}$ and $T_{0}=I$;

\noindent(ii) $\lim_{t\rightarrow0}||T_{t}f-f||_{L^{\infty}}=0$ for any $f\in
C_{0}(\mathbf{%
\mathbb{Q}
}_{p}^{n},\mathbb{R})$;

\noindent(iii) $0\leq T_{t}f\leq1$ if $0\leq f\leq1$, with $f\in
C_{0}(\mathbf{%
\mathbb{Q}
}_{p}^{n},\mathbb{R})$ and for any $t\geq0$.
\end{definition}

Theorem \ref{Hille-Yosida-Ray} characterizes the Feller semigroups, more
precisely, if $(A,Dom(A))$ satisfies Theorem \ref{Hille-Yosida-Ray}, then $A$
has a closed extension which is the generator of a Feller semigroup.

\subsection{Pseudodifferential operators attached to negative definite
functions}

\begin{remark}
Let $\boldsymbol{\psi}_{j}:\mathbf{%
\mathbb{Q}
}_{p}^{n}\rightarrow%
\mathbb{R}
_{+}$ be radial, continuous, negative definite functions of types 0, 1 or 2,
for $j=1,\ldots,m,$ and let $b_{j}$ be positive real numbers, for
$j=1,\ldots,m$. To these functions we attach the following symbol:
\begin{equation}
p(\xi):=\sum_{j=1}^{m}b_{j}\boldsymbol{\psi}_{j}(||\xi||_{p})\label{p(x,t)}%
\end{equation}
and the pseudodifferential operator
\[
(P(\partial)\varphi)(x)=-%
\mathcal{F}%
_{\xi\rightarrow x}^{-1}(p(\xi)%
\mathcal{F}%
_{x\rightarrow\xi}\varphi)
\]
for $\varphi\in\mathcal{D}_{\mathbb{R}}(%
\mathbb{Q}
_{p}^{n})$. Notice that $(P(\partial)\varphi)(x)=-\sum_{j=1}^{m}%
b_{j}(D_{\boldsymbol{\psi}_{j}}\varphi)(x)$, where $(D_{\boldsymbol{\psi}_{j}%
}\varphi)(x)\allowbreak=%
\mathcal{F}%
_{\xi\rightarrow x}^{-1}(\boldsymbol{\psi}_{j}(||\xi||_{p})%
\mathcal{F}%
_{x\rightarrow\xi}\varphi)$ for $\varphi\in\mathcal{D}_{\mathbb{R}}(%
\mathbb{Q}
_{p}^{n})$. Moreover, $P(\partial):\mathcal{D}_{\mathbb{R}}(%
\mathbb{Q}
_{p}^{n})\rightarrow L^{2}(%
\mathbb{Q}
_{p}^{n})\cap C(%
\mathbb{Q}
_{p}^{n})$.
\end{remark}

\begin{remark}
Notice that $p(\xi)$ defines a real-valued, negative definite function, and
thus $%
\mathcal{F}%
(u_{t}):=e^{-tp(\xi)}$, $t>0$, $\xi\in%
\mathbb{Q}
_{p}^{n}$, is a convolution semigroup of measures, see \cite[Theorem
8.3]{Berg-Gunnar}. The condition $p(0)=0$ implies that $(u_{t})_{t>0}$ are
probability measures, see \cite[Corollary 8.6]{Berg-Gunnar}.
\end{remark}

\begin{remark}
\label{Obs p(x,z,t)}(i) If all the functions $\boldsymbol{\psi}_{j}$ appearing
in (\ref{p(x,t)}) are of types 0 or 1 and there is at least one function
$\boldsymbol{\psi}_{j}$ of type 1, then $p(\xi)$ is a negative definite,
continuous and radial function in $\xi$ of type 1.

\noindent(ii) If all the functions $\boldsymbol{\psi}_{j}$ appearing in
(\ref{p(x,t)}) are of types 0, 1 or 2 and there is at least one function
$\boldsymbol{\psi}_{j}$ of type 2,\ then $p(\xi)$ is a negative definite,
continuous and radial function in $\xi$ of type 2.

\noindent(iii) In the case (i), there are positive constants $C,$ $\beta$,
with $\beta\geq1$, such that
\[
p(\xi)\leq C[\max\{1,||\xi||_{p}\}]^{\beta}.
\]
We set $\boldsymbol{\psi}(||\xi||_{p}):=||\xi||_{p}^{\beta}$. Notice that
$\max\{1,\boldsymbol{\psi}(||\xi||_{p})\}=[\max\{1,||\xi||_{p}\}]^{\beta}$. We
attach to $P(\partial)$ the space $B_{\boldsymbol{\psi},\infty}(\mathbb{R})$.
In the case (ii),
\[
p(\xi)\leq C\sum_{j\in J}\boldsymbol{\psi}_{j}(||\xi||_{p})=:\boldsymbol{\psi
}(||\xi||_{p}),
\]
where $J$ is the set of indices $j\in\{1,2,\ldots,m\}$ for which
$\boldsymbol{\psi}_{j}$ is of type 2. We attach to $P(\partial)$ the space
$B_{\boldsymbol{\psi},\infty}(\mathbb{R})$.

\noindent(iv) From now on we will assume that in (\ref{p(x,t)}) there is at
least one function $\boldsymbol{\psi}_{j}$ of type 1 or type 2.
\end{remark}

\begin{remark}
\label{operator P(D) real-valued}If $\varphi\in\mathcal{D}_{\mathbb{R}}(%
\mathbb{Q}
_{p}^{n})$, then $(P(\partial)\varphi)(x)$ is a real-valued function, i.e.
$(P(\partial)\varphi)\allowbreak(x)=\overline{(P(\partial)\varphi)(x)}$.
\end{remark}

\begin{lemma}
\label{operator P(D)}With the conventions and notations introduced in Remark
\ref{Obs p(x,z,t)}, the mapping $P(\partial):B_{\boldsymbol{\psi},\infty
}\left(  \mathbb{R}\right)  \rightarrow B_{\boldsymbol{\psi},\infty}\left(
\mathbb{R}\right)  $ is a well-defined continuous operator.
\end{lemma}

\begin{proof}
In the cases (i)-(ii) of Remark \ref{Obs p(x,z,t)}, for $\varphi\in
\mathcal{D}_{\mathbb{R}}\left(
\mathbb{Q}
_{p}^{n}\right)  $, we have
\begin{equation}
||P(\partial)\varphi||_{\boldsymbol{\psi},l}\leq C||\varphi
||_{\boldsymbol{\psi},l+2}, \label{Des p(x,z,t)}%
\end{equation}
which implies (by using Remark \ref{operator P(D) real-valued} and the fact
that $\mathcal{D}_{\mathbb{R}}\left(
\mathbb{Q}
_{p}^{n}\right)  $ is dense in $B_{\boldsymbol{\psi},l}\left(  \mathbb{R}%
\right)  $ for any $l$) that $P(\partial):B_{\boldsymbol{\psi},l+2}\left(
\mathbb{R}\right)  \rightarrow B_{\boldsymbol{\psi},l}\left(  \mathbb{R}%
\right)  $ is a well-defined continuous mapping for any $l\in%
\mathbb{N}
$, and consequently $P(\partial):B_{\boldsymbol{\psi},\infty}\left(
\mathbb{R}\right)  \rightarrow B_{\boldsymbol{\psi},\infty}\left(
\mathbb{R}\right)  $ is a well-defined operator. The continuity is established
by using an argument based on sequences and (\ref{Des p(x,z,t)}).
\end{proof}

\begin{example}
\label{Example5}Take $\beta>0$ and set $\boldsymbol{\psi}_{\beta}\left(
\left\Vert \xi\right\Vert _{p}\right)  :=\left\Vert \xi\right\Vert _{p}%
^{\beta}$. Then the operator $D_{\boldsymbol{\psi}_{\beta}}$ is the Taibleson
operator $D_{T}^{\beta}$, see Example \ref{example neg def}, which admits the
following representation:
\[
\left(  D_{\boldsymbol{\psi}_{\beta}}\varphi\right)  \left(  x\right)
=\frac{1-p^{\beta}}{1-p^{-\beta-n}}\int\limits_{%
\mathbb{Q}
_{p}^{n}}\left\Vert y\right\Vert _{p}^{-\beta-n}\left(  \varphi\left(
x-y\right)  -\varphi\left(  x\right)  \right)  d^{n}y
\]
for $\varphi\in\mathcal{D}_{\mathbb{R}}\left(
\mathbb{Q}
_{p}^{n}\right)  $, see \cite{R-Zu}. Notice that $-D_{\boldsymbol{\psi}%
_{\beta}}$ satisfies Remark \ref{operator P(D) real-valued}\ and the positive
maximum principle on $\mathcal{D}_{\mathbb{R}}\left(
\mathbb{Q}
_{p}^{n}\right)  $. On the other hand, $-D_{\boldsymbol{\psi}_{\beta}}%
\varphi=-f_{-\beta}\ast\varphi$, where $f_{-\beta}\in\mathcal{D}_{\mathbb{R}%
}^{\prime}\left(
\mathbb{Q}
_{p}^{n}\right)  $ is the Riesz kernel. By identifying $f_{-\beta}$ with a
positive measure, we have
\[
\left\Vert -D_{\boldsymbol{\psi}_{\beta}}\varphi\right\Vert _{L^{\infty}%
}=\left\Vert -f_{-\beta}\ast\varphi\right\Vert _{L^{\infty}}\leq\left\Vert
-f_{-\beta}\right\Vert \left\Vert \varphi\right\Vert _{L^{\infty}},
\]
where \ $\left\Vert -f_{-\beta}\right\Vert <\infty$ is the total variation of
$-f_{-\beta}$. This implies that $-D_{\boldsymbol{\psi}_{\beta}}\varphi$ has a
unique continuous extension to $C_{0}\left(
\mathbb{Q}
_{p}^{n},\mathbb{R}\right)  $, which satisfies the positive maximum principle.
\end{example}

\begin{example}
\label{Example6}Set $J$ as in Example \ref{Example2}. Then the function
$\boldsymbol{\psi}_{J}\left(  \left\Vert \xi\right\Vert _{p}\right)
:=1-\widehat{J}(||\xi||_{p})$ is negative definite and it satisfies
$0\leq1-\widehat{J}(||\xi||_{p})\leq2$ for $\xi\in%
\mathbb{Q}
_{p}^{n}$, which implies that $\boldsymbol{\psi}_{J}\left(  \left\Vert
\xi\right\Vert _{p}\right)  $ is of type 0. Then%
\begin{align*}
\left(  D_{\boldsymbol{\psi}_{J}}\varphi\right)  \left(  x\right)   &
=\mathcal{F}_{\xi\rightarrow x}^{-1}\left(  \left\{  1-\widehat{J}(||\xi
||_{p})\right\}  \mathcal{F}_{x\rightarrow\xi}\varphi\right) \\
&  =-\int\limits_{%
\mathbb{Q}
_{p}^{n}}J\left(  \left\Vert x-y\right\Vert _{p}\right)  \left(
\varphi\left(  y\right)  -\varphi\left(  x\right)  \right)  d^{n}y
\end{align*}
for $\varphi\in\mathcal{D}_{\mathbb{R}}\left(
\mathbb{Q}
_{p}^{n}\right)  $. Notice that $-D_{\boldsymbol{\psi}_{\beta}}$ satisfies
Remark \ref{operator P(D) real-valued}\ and the positive maximum principle.
\ Since%
\[
\left\Vert D_{\boldsymbol{\psi}_{J}}\varphi\right\Vert _{L^{\infty}}%
\leq\left(  1+\left\Vert J\right\Vert _{L^{1}}\right)  \left\Vert
\varphi\right\Vert _{L^{\infty}},
\]
we have $D_{\boldsymbol{\psi}_{J}}\varphi$ has a unique continuous extension
to $C_{0}\left(
\mathbb{Q}
_{p}^{n},\mathbb{R}\right)  $, which satisfies the positive maximum principle.
\end{example}

\begin{example}
\label{Example7}Take $\boldsymbol{\psi}_{\beta_{i}}\left(  \left\Vert
\xi\right\Vert _{p}\right)  $ for $i=1,\ldots,l$ with $0<\beta_{1}%
<\cdots<\beta_{l}$ as in Example \ref{Example5}, and let $J_{i}$ be functions
as in Example \ref{Example2} for $i=l+1,\ldots,m$, and set $\boldsymbol{\psi
}_{J_{i}}\left(  \left\Vert \xi\right\Vert _{p}\right)  :=1-\widehat{J_{i}%
}(||\xi||_{p})$ as in Example \ref{Example6}. Then
\[
\boldsymbol{\psi}(||\xi||_{p}):=\sum_{i=1}^{l}b_{i}\boldsymbol{\psi}%
_{\beta_{i}}\left(  \left\Vert \xi\right\Vert _{p}\right)  +\sum_{i=l+1}%
^{m}b_{i}\boldsymbol{\psi}_{J_{i}}\left(  \left\Vert \xi\right\Vert
_{p}\right)  ,
\]
where the $b_{i}$'s are as before, is a negative definite function of type 1,
and it satisfies $\boldsymbol{\psi}(||\xi||_{p})\leq C\max\left\{
1,||\xi||_{p}\right\}  ^{\beta_{l}}$. Then the operator%
\[
P(\partial)=-\sum_{i=1}^{l}b_{i}D_{\boldsymbol{\psi}_{\beta_{i}}}-\sum
_{i=l+1}^{m}b_{i}D_{\boldsymbol{\psi}_{J_{i}}}%
\]
satisfies Remark \ref{operator P(D) real-valued} and the positive maximum principle.
\end{example}

\begin{lemma}
\label{Lambda+P}For any fixed positive real number $\lambda$, the equation
\begin{equation}
(\lambda-P(\partial))u=f,\text{ }f\in B_{\boldsymbol{\psi},\infty}%
(\mathbb{R}), \label{equation lambda}%
\end{equation}
has a unique solution $u$ in $B_{\boldsymbol{\psi},\infty}(\mathbb{R})$.
\end{lemma}

\begin{proof}
By using that $f\in L^{2}$, we have
\[
\widehat{u}\left(  \xi;\lambda\right)  =\frac{\widehat{f}(\xi)}{\lambda
+\sum_{j=1}^{m}b_{j}\boldsymbol{\psi}_{j}(||\xi||_{p})}:=\frac{\widehat{f}%
(\xi)}{\lambda+H(||\xi||_{p})}.
\]
We now recall that $\boldsymbol{\psi}_{j}(||\xi||_{p})\geq0$ for any $j$, see
\cite[Proposition 7.5 and p. 39]{Berg-Gunnar}, then
\[
\inf_{\xi}H(||\xi||_{p})\geq0,
\]
and consequently $\left\vert \widehat{u}\left(  \xi;\lambda\right)
\right\vert \leq\frac{\left\vert \widehat{f}(\xi)\right\vert }{\lambda}$. By
Lemma \ref{Lemma1} (ii)-(iii), $u\left(  x;\lambda\right)  \in
B_{\boldsymbol{\psi},\infty}(\mathbb{C})$.\ The uniqueness follows from the
fact that $(\lambda-P(\partial))u=0$ has $u=0$ as a unique solution since
$B_{\boldsymbol{\psi},\infty}(\mathbb{C})\subset C_{0}(%
\mathbb{Q}
_{p}^{n},\mathbb{C})$, see Lemma \ref{Lemma1} (v). We now show that $u\left(
x;\lambda\right)  $ is a real valued function. Recall that
\[
u\left(  x;\lambda\right)  =\lim_{k\rightarrow\infty}u_{k}\left(
x;\lambda\right)  \text{ in }L^{2}\text{-sense,}%
\]
where
\[
u_{k}\left(  x;\lambda\right)  =\int\limits_{||\xi||_{p}\leq p^{k}}\frac
{\chi_{p}\left(  -\xi\cdot x\right)  \widehat{f}(\xi)}{\lambda+H(||\xi||_{p}%
)}d^{n}\xi.
\]
\textbf{Claim }$u\left(  x;\lambda\right)  =\lim_{k\rightarrow\infty}%
u_{k}\left(  x;\lambda\right)  $ in $L^{2}$-sense, where the \ $u_{k}\left(
x;\lambda\right)  $'s are real valued functions.

Then there exists a subsequence of $\left\{  u_{k}\left(  x;\lambda\right)
\right\}  _{k\in\mathbb{N}}$ converging almost uniformly to the same limit,
which implies that $u\left(  x;\lambda\right)  $ is a real valued function
outside of a zero measure subset of $%
\mathbb{Q}
_{p}^{n}$. Since $u\left(  x;\lambda\right)  $ is a continuous function in
$x$, $B_{\boldsymbol{\psi},\infty}(\mathbb{C})\subset C_{0}(%
\mathbb{Q}
_{p}^{n},\mathbb{C})$, necessarily $u\left(  x;\lambda\right)  $ is a
continuous function at every $x\in%
\mathbb{Q}
_{p}^{n}$.

\textbf{Proof of the} \textbf{Claim. }It is sufficient to show that
$u_{k}\left(  x;\lambda\right)  =\overline{u_{k}\left(  x;\lambda\right)  }$
for any $k$. Indeed,
\begin{align*}
\overline{u_{k}\left(  x;\lambda\right)  }  &  =\int\limits_{||\xi||_{p}\leq
p^{k}}\frac{\chi_{p}\left(  \xi\cdot x\right)  \overline{\widehat{f}(\xi}%
)}{\lambda+H(||\xi||_{p})}d^{n}\xi=\int\limits_{||\xi||_{p}\leq p^{k}}%
\frac{\chi_{p}\left(  \xi\cdot x\right)  \left(  \mathcal{F}^{-1}f\right)
(\xi)}{\lambda+H(||\xi||_{p})}d^{n}\xi\\
&  =\int\limits_{||\xi||_{p}\leq p^{k}}\frac{\chi_{p}\left(  -\xi\cdot
x\right)  \widehat{f}(\xi)}{\lambda+H(||\xi||_{p})}d^{n}\xi,
\end{align*}
where we used that $\overline{\widehat{f}(\xi})=\left(  \mathcal{F}%
^{-1}f\right)  (\xi)$ and $\left(  \mathcal{F}^{-1}f\right)  (-\xi
)=\widehat{f}(\xi)$ for $f\in L^{2}$.
\end{proof}

\begin{theorem}
\label{Feller semigroups} Assume that $(P(\partial),B_{\boldsymbol{\psi
},\infty}\left(  \mathbb{R}\right)  )$ satisfies the positive maximum
principle. Then the closure $(\overline{P(\partial)},Dom(\overline
{P(\partial)})$ of $(P(\partial),B_{\boldsymbol{\psi},\infty}\left(
\mathbb{R}\right)  )$ on $C_{0}(%
\mathbb{Q}
_{p}^{n},\mathbb{R})$ is the generator of a Feller semigroup.
\end{theorem}

\begin{proof}
We show that $(P(\partial),B_{\boldsymbol{\psi},\infty}\left(  \mathbb{R}%
\right)  )$ satisfies conditions (i) and (iii) in the Hille-Yosida-Ray's
theorem, see Theorem \ref{Hille-Yosida-Ray}. By Lemma \ref{operator P(D)},
$P(\partial):B_{\boldsymbol{\psi},\infty}\left(  \mathbb{R}\right)
\rightarrow B_{\boldsymbol{\psi},\infty}\left(  \mathbb{R}\right)  $ is a
well-defined continuous operator, and by Lemma \ref{Lemma1}-(vi),
$B_{\boldsymbol{\psi},\infty}\left(  \mathbb{R}\right)  $ is densely and
continuously embedded in $C_{0}(%
\mathbb{Q}
_{p}^{n},\mathbb{R})$. For the third condition, Lemma \ref{Lambda+P} implies
that the rank of $\lambda-P(\partial)$ is $B_{\boldsymbol{\psi},\infty}\left(
\mathbb{R}\right)  $ which is dense in $C_{0}(%
\mathbb{Q}
_{p}^{n},\mathbb{R})$.
\end{proof}

\section{\label{Parabolic-Type Equations}Parabolic-Type Equations}

Let $T>0$ and let $f(x,t):%
\mathbb{Q}
_{p}^{n}\times\lbrack0,T]\rightarrow%
\mathbb{R}
$ such that $f(x,\cdot):[0,T]\rightarrow C_{0}(%
\mathbb{Q}
_{p}^{n},\mathbb{R})$, let $b_{j}$ be positive real numbers, for
$j=1,\ldots,m$, and consider
\[%
\begin{array}
[c]{cccc}%
P(\partial): & B_{\boldsymbol{\psi},\infty}\left(  \mathbb{R}\right)  &
\rightarrow & B_{\boldsymbol{\psi},\infty}\left(  \mathbb{R}\right) \\
&  &  & \\
& g & \longrightarrow & -\sum_{j=1}^{m}b_{j}D_{\boldsymbol{\psi}_{j}}g.
\end{array}
\]
Our aim is to study the following initial value problem:%
\begin{equation}
\left\{
\begin{array}
[c]{l}%
u(x,\cdot)\in C([0,T],B_{\boldsymbol{\psi},\infty}\left(  \mathbb{R}\right)
)\cap C^{1}([0,T],C_{0}(%
\mathbb{Q}
_{p}^{n},\mathbb{R}));\\
\\
\frac{\partial u}{\partial t}(x,t)=P(\partial)u(x,t)+f(x,t)\text{, \ }%
t\in\lbrack0,T]\text{,\ }x\in%
\mathbb{Q}
_{p}^{n};\\
\\
u(x,0)=h(x)\in B_{\boldsymbol{\psi},\infty}\left(  \mathbb{R}\right)  .
\end{array}
\right.  \label{IVP}%
\end{equation}

The proof of the following lemma is included for the sake of completeness.

\begin{lemma}
\label{m-dissipative}The operator $\overline{P(\partial)}:Dom(\overline
{P(\partial))}\rightarrow C_{0}(%
\mathbb{Q}
_{p}^{n},\mathbb{R})$ is $m-$dissipative.
\end{lemma}

\begin{proof}
We first verify that $\left(  \overline{P(\partial)},Dom(\overline
{P(\partial)})\right)  $ is dissipative, i.e.
\begin{equation}
\left\Vert \left(  1-\lambda\overline{P(\partial)}\right)  u\right\Vert
_{L^{\infty}}\geq\left\Vert u\right\Vert _{L^{\infty}} \label{Eq16}%
\end{equation}
for $u\in Dom(\overline{P(\partial))}$. Take functions $u$, $v$ such that
$\overline{P(\partial)}u=v$, then there exist a sequences $\left\{
u_{m}\right\}  _{m\in\mathbb{N}}$ in $B_{\boldsymbol{\psi},\infty}\left(
\mathbb{R}\right)  $ and $\left\{  v_{m}\right\}  _{m\in\mathbb{N}}$ in
$C_{0}(%
\mathbb{Q}
_{p}^{n},\mathbb{R})$ such that $u_{m}\rightarrow u$, $v_{m}\rightarrow v$,
and $P(\partial)u_{m}=v_{m}$. Since $P(\partial)$ is dissipative, cf.
\cite[Chapter 4, Lemma 2.1]{E-K}, we have
\begin{equation}
\left\Vert \left(  1-\lambda P(\partial)\right)  u_{m}\right\Vert _{L^{\infty
}}\geq\left\Vert u_{m}\right\Vert _{L^{\infty}}. \label{Eq17}%
\end{equation}
Now (\ref{Eq16}) follows from (\ref{Eq17}) by taking the limit $m\rightarrow
\infty$.

To show that $\left(  \overline{P(\partial)},Dom(\overline{P(\partial
))}\right)  $ is $m-$dissipative, we show that there exists $\lambda>0$ such
that for all $f\in C_{0}(%
\mathbb{Q}
_{p}^{n},\mathbb{R})$, there exists a solution $u\in Dom\left(  \overline
{P(\partial)}\right)  $ of $\left(  1-\lambda\overline{P(\partial)}\right)
u=f$, cf. \cite[Proposition 2.2.6]{C-H}. Since $B_{\boldsymbol{\psi},\infty
}(\mathbb{R})$ is dense in $C_{0}(%
\mathbb{Q}
_{p}^{n},\mathbb{R})$, there exists a sequence $\left\{  f_{m}\right\}
_{m\in\mathbb{N}}$ in $B_{\boldsymbol{\psi},\infty}(\mathbb{R})$ such that
$f_{m}$ $\underrightarrow{||\cdot||_{\boldsymbol{\psi},l}}$ $f$, for any
$l\in\mathbb{N}$, now by using the density of $\mathcal{D}_{\mathbb{R}}\left(
%
\mathbb{Q}
_{p}^{n}\right)  $\ in $B_{\boldsymbol{\psi},\infty}(\mathbb{R})$, there
exists a sequence $\left\{  g_{m}\right\}  _{m\in\mathbb{N}}$ in
$\mathcal{D}_{\mathbb{R}}\left(
\mathbb{Q}
_{p}^{n}\right)  $ such that $\left\Vert f_{m}-g_{m}\right\Vert
_{\boldsymbol{\psi},l}\leq\frac{1}{m}$, thus $g_{m}$ $\underrightarrow
{||\cdot||_{\boldsymbol{\psi},l}}$ $f$ and since $B_{\boldsymbol{\psi},\infty
}(\mathbb{R})$ $\hookrightarrow$ $C_{0}(%
\mathbb{Q}
_{p}^{n},\mathbb{R})$ by Lemma \ref{Lemma1} (vi), we get that $g_{m}$
$\underrightarrow{||\cdot||_{L^{\infty}}}$ $f$ . For each $g_{m}$, there
exists $u_{m}\in B_{\boldsymbol{\psi},\infty}\left(  \mathbb{R}\right)  $ such
that for any $\lambda>0$, $\left(  1-\lambda^{-1}P(\partial)\right)  \lambda
u_{m}=g_{m}$, i.e. $P(\partial)u_{m}=\lambda u_{m}-g_{m}$, cf. Lemma
\ref{Lambda+P}.

\textbf{Claim.} The sequence $\left\{  u_{m}\right\}  _{m\in\mathbb{N}}$ is
Cauchy in $B_{\boldsymbol{\psi},l}(\mathbb{R})$ for any $l\in\mathbb{N}$, i.e.
is Cauchy in $B_{\boldsymbol{\psi},\infty}(\mathbb{R})$.

By the Claim, $u_{m}\rightarrow u\in B_{\boldsymbol{\psi},\infty}\left(
\mathbb{R}\right)  \hookrightarrow C_{0}(%
\mathbb{Q}
_{p}^{n},\mathbb{R})$ for some $u$. Therefore, $\overline{P(\partial
)}u=\lambda u-f$, i.e. $\left(  1-\lambda^{-1}\overline{P(\partial)}\right)
\lambda u=f$.

\textbf{Proof of the Claim.}

By using that
\[
\widehat{u_{m}}\left(  \xi;\lambda\right)  =\frac{\widehat{g_{m}}(\xi
)}{\lambda+\sum_{j=1}^{m}b_{j}\boldsymbol{\psi}_{j}(||\xi||_{p})},
\]
we have%
\[
\left\vert \widehat{u_{m}}\left(  \xi;\lambda\right)  -\widehat{u_{n}}\left(
\xi;\lambda\right)  \right\vert \leq\frac{\left\vert \widehat{g_{m}}%
(\xi)-\widehat{g_{n}}(\xi)\right\vert }{\lambda},
\]
which implies that $\left\Vert u_{m}\left(  \xi;\lambda\right)  -u_{n}\left(
\xi;\lambda\right)  \right\Vert _{\boldsymbol{\psi},l}\leq\frac{1}{\lambda
}||g_{m}(\xi)-g_{n}(\xi)||_{\boldsymbol{\psi},l}$ for any $l\in\mathbb{N}$,
and since $g_{m}$ $\underrightarrow{||\cdot||_{\boldsymbol{\psi},l}}$ $f$, for
any $l\in\mathbb{N}$, the sequence $\left\{  u_{m}\right\}  _{m\in\mathbb{N}}$
is Cauchy in $B_{\boldsymbol{\psi},\infty}(\mathbb{R})$.
\end{proof}

\begin{lemma}
\label{Lemma 8}Assume that at lest one of the following conditions hold:

\noindent(i) $f(x,\cdot)\in L^{1}((0,T),B_{\boldsymbol{\psi},\infty}\left(
\mathbb{R}\right)  )$;

\noindent(ii) $f(x,\cdot)\in W^{1,1}((0,T),C_{0}(%
\mathbb{Q}
_{p}^{n},\mathbb{R}))$.

Then the initial value problem (\ref{IVP}) has a unique solution of the form
\[
u(x,t)=T(t)h(x)+%
{\displaystyle\int\nolimits_{0}^{t}}
T(t-s)f(x,s)ds\text{, for all }t\in\lbrack0,T]\text{,\ }%
\]
where $(T(t))_{t\geq0}$ is the Feller semigroup associated to the operator
$\overline{P(\partial)}$.
\end{lemma}

\begin{proof}
It follows from \cite[Lemma 4.1.1]{C-H}, \cite[Corollary 4.1.2]{C-H} and
\cite[Proposition 4.1.6]{C-H}, by using the fact that $\overline{P(\partial)}$
is $m$-dissipative.
\end{proof}

Consider the following Cauchy problem:%

\begin{equation}
\left\{
\begin{array}
[c]{ll}%
u(x,\cdot)\in C([0,T],B_{\boldsymbol{\psi},\infty}\left(  \mathbb{R}\right)
)\cap C^{1}([0,T],C_{0}(%
\mathbb{Q}
_{p}^{n},\mathbb{R})); & \\
& \\
\frac{\partial u}{\partial t}(x,t)=P(\partial)u(x,t)\text{,} & \\
& \\
u(x,0)=u_{0}(x)\in\mathcal{D}\left(
\mathbb{Q}
_{p}^{n},\mathbb{R}\right)  . &
\end{array}
\right.  \label{Cauchy problem}%
\end{equation}
We define $p(\xi)=\sum_{j=1}^{m}b_{j}D_{\boldsymbol{\psi}_{j}}g$,
\[
u(x,t)=\int_{\mathbf{%
\mathbb{Q}
}_{p}^{n}}\chi_{p}\left(  -x\cdot\xi\right)  e^{-tp(\xi)}\widehat{u_{0}}%
(\xi)d^{n}\xi\text{, for }x\in%
\mathbb{Q}
_{p}^{n}\text{ and }t\geq0\text{,}%
\]
and
\[
Z(x,t)=%
\mathcal{F}%
_{\xi\rightarrow x}^{-1}(e^{-tp(\xi)})\text{ in }\mathcal{D}^{\prime}\left(
\mathbb{Q}
_{p}^{n}\right)  \text{, for }x\in%
\mathbb{Q}
_{p}^{n}\text{ and }t\geq0\text{.}%
\]
Notice that $u(x,t)=%
\mathcal{F}%
_{\xi\rightarrow x}^{-1}(e^{-tp(\xi)})\ast u_{0}(x)$ in $\mathcal{D}^{\prime
}\left(
\mathbb{Q}
_{p}^{n}\right)  $, for $x\in%
\mathbb{Q}
_{p}^{n}$ and $t\geq0$.

\begin{lemma}
\label{Lemma 9}The function $u(x,t)$ defined above satisfies the following conditions:

\noindent(C1) $u(x,\cdot)\in C([0,T],B_{\boldsymbol{\psi},\infty}\left(
\mathbb{R}\right)  )\cap C^{1}([0,T],C_{0}(%
\mathbb{Q}
_{p}^{n}))$ and the derivative is given by
\begin{equation}
\frac{\partial u}{\partial t}(x,t)=-\int_{\mathbf{%
\mathbb{Q}
}_{p}^{n}}\chi_{p}\left(  -x\cdot\xi\right)  p(\xi)e^{-tp(\xi)}\widehat{u_{0}%
}(\xi)d^{n}\xi; \label{Eq19}%
\end{equation}

\noindent(C2) $u(x,\cdot)\in L^{1}(%
\mathbb{Q}
_{p}^{n})\cap L^{2}(%
\mathbb{Q}
_{p}^{n})$ for any $t\geq0$, and
\begin{equation}
P(\partial)u(x,t)=-\int_{\mathbf{%
\mathbb{Q}
}_{p}^{n}}\chi_{p}\left(  -x\cdot\xi\right)  p(\xi)e^{-tp(\xi)}\widehat{u_{0}%
}(\xi)d^{n}\xi. \label{Eq20}%
\end{equation}
Furthermore $u(x,\cdot)$ is a solution of the initial value problem
(\ref{Cauchy problem}).
\end{lemma}

\begin{proof}
The result is proved through the following claims:

\textbf{Claim 1.} $u(x,\cdot)\in C([0,T],B_{\boldsymbol{\psi},\infty}\left(
\mathbb{R}\right)  )$

We first show that $u(\cdot,t)\in B_{\boldsymbol{\psi},\infty}\left(
\mathbb{R}\right)  $ for all $t\geq0$. By using that $u(x,t)=%
\mathcal{F}%
_{\xi\rightarrow x}^{-1}(e^{-tp(\xi)})\ast u_{0}(x)$ in $\mathcal{D}^{\prime
}\left(
\mathbb{Q}
_{p}^{n}\right)  $. This convolution exists because $u_{0}(x)$ has compact
support, then
\begin{equation}
\widehat{u}(\xi,t)=e^{-tp(\xi)}\widehat{u_{0}}(\xi)\text{ in }\mathcal{D}%
^{\prime}\left(
\mathbb{Q}
_{p}^{n}\right)  \label{Eq21}%
\end{equation}
for $t\geq0$. Now%
\begin{align*}
||u(\cdot,t)||_{\boldsymbol{\psi},l}^{2}  &  =\int_{\mathbf{%
\mathbb{Q}
}_{p}^{n}}[\max(1,\boldsymbol{\psi}(||\xi||_{p}))]^{l}|\widehat{u}(\xi
,t)|^{2}d^{n}\xi\\
&  \leq\int_{\mathbf{%
\mathbb{Q}
}_{p}^{n}}[\max(1,\boldsymbol{\psi}(||\xi||_{p}))]^{l}|\widehat{u_{0}}%
(\xi)|^{2}d^{n}\xi=||u_{0}||_{\boldsymbol{\psi},l}^{2},
\end{align*}
i.e. $||u(\cdot,t)||_{\boldsymbol{\psi},l}\leq||u_{0}||_{\boldsymbol{\psi},l}%
$, for any $l\in\mathbb{N}$, which implies that $u(\cdot,t)\in
B_{\boldsymbol{\psi},\infty}\left(  \mathbb{R}\right)  $ for all $t\geq0$.

We now verify that
\[
\lim_{t\rightarrow t_{0}}||u(x,t)-u(x,t_{0})||_{\boldsymbol{\psi},l}%
^{2}=0\text{ for any }l\in\mathbb{N}\text{,}%
\]
which implies the continuity of $u(\cdot,t)$. The verification of this fact is
done by using (\ref{Eq21}) and the dominated convergence theorem.

\textbf{Claim 2. }$u(x,\cdot)\in C^{1}([0,T],C_{0}(%
\mathbb{Q}
_{p}^{n}))$ and $\frac{\partial u}{\partial t}(x,t)$ is given by (\ref{Eq19}).

Set
\[
h_{t}(x):=\int_{\mathbf{%
\mathbb{Q}
}_{p}^{n}}\chi_{p}\left(  -x\cdot\xi\right)  p(\xi)e^{-tp(\xi)}\widehat{u_{0}%
}(\xi)d^{n}\xi\text{, for }x\in\mathbf{%
\mathbb{Q}
}_{p}^{n}\text{ and }t\geq0.
\]
Notice that for$\ $any $t\geq0$ fixed, and any $x\in\mathbf{%
\mathbb{Q}
}_{p}^{n}$ fixed, $p(\xi)e^{-tp(\xi)}\widehat{u_{0}}(\xi)$ is an integrable
function in $\xi$, and thus by the Riemann-Lebesgue theorem $h_{t}(x)\in
C_{0}(%
\mathbb{Q}
_{p}^{n})$ for$\ $any $t\geq0$ fixed. Now, by applying the mean value theorem
we have
\[
\frac{e^{-tp(\xi)}-e^{-t_{0}p(\xi)}}{t-t_{0}}=-p(\xi)e^{-\tau(x)p(\xi
)}\text{,}%
\]
for some $\tau(x)$ between $t$ and $t_{0}$. So that
\begin{multline*}
\lim_{t\rightarrow t_{0}}\left\Vert \frac{u(x,t)-u(x,t_{0})}{t-t_{0}}%
+h_{t}(x)\right\Vert _{L^{\infty}}\\
=\lim_{t\rightarrow t_{0}}\left\Vert \frac{\int_{\mathbf{%
\mathbb{Q}
}_{p}^{n}}\chi_{p}\left(  -x\cdot\xi\right)  \left\{  e^{-tp(\xi)}%
-e^{-t_{0}p(\xi)}\right\}  \widehat{u_{0}}(\xi)d^{n}\xi}{t-t_{0}}%
+h_{t}(x)\right\Vert _{L^{\infty}}\\
=\lim_{t\rightarrow t_{0}}\left\Vert -\int_{\mathbf{%
\mathbb{Q}
}_{p}^{n}}\chi_{p}\left(  -x\cdot\xi\right)  p(\xi)e^{-\tau(x)p(\xi)}%
\widehat{u_{0}}(\xi)d^{n}\xi+h_{t}(x)\right\Vert _{L^{\infty}}\\
=\lim_{t\rightarrow t_{0}}\left\Vert -\int_{\mathbf{%
\mathbb{Q}
}_{p}^{n}}\chi_{p}\left(  -x\cdot\xi\right)  p(\xi)\widehat{u_{0}}(\xi)\left[
e^{-\tau(x)p(\xi)}-e^{-tp(\xi)}\right]  d^{n}\xi\right\Vert _{L^{\infty}}%
\end{multline*}
Now, when $t\rightarrow t_{0},$ $\tau(x)\rightarrow t,$ so that
\[
\lim_{t\rightarrow t_{0}}\left\Vert \frac{u(x,t)-u(x,t_{0})}{t-t_{0}}%
+h_{t}(x)\right\Vert _{L^{\infty}}=0.
\]

\textbf{Claim 3. }The assertion $(C_{2})$ holds.

Indeed, since $u_{0}$ has compact support
\[
e^{-tp(\xi)}\widehat{u_{0}},\text{ }p(\xi)e^{-tp(\xi)}\widehat{u_{0}}(\xi)\in
L^{1}(%
\mathbb{Q}
_{p}^{n})\cap L^{2}(%
\mathbb{Q}
_{p}^{n}),
\]
and thus $P(\partial)u(x,t)=-%
\mathcal{F}%
_{\xi\rightarrow x}^{-1}(p(\xi)%
\mathcal{F}%
_{x\rightarrow\xi}u(x,t)).$
\end{proof}

\begin{theorem}
\label{Theorem 3}Assuming that operator $P(\partial)$ satisfies the positive
maximum principle. Then the Cauchy problem \
\begin{equation}
\left\{
\begin{array}
[c]{l}%
u(x,\cdot)\in C([0,T],B_{\boldsymbol{\psi},\infty}\left(  \mathbb{R}\right)
)\cap C^{1}([0,T],C_{0}(%
\mathbb{Q}
_{p}^{n},\mathbb{R}));\\
\\
\frac{\partial u}{\partial t}(x,t)=P(\partial)u(x,t)+f(x,t)\\
\\
u(x,0)=u_{0}(x)\in B_{\boldsymbol{\psi},\infty}\left(  \mathbb{R}\right)  ,
\end{array}
\right.  \label{IVP2}%
\end{equation}
where $f(x,t)$ is a function satisfying the assumptions of Lemma
\ref{Lemma 8}, has a unique solution given by
\[
u(x,t)=Z_{t}(x)\ast u_{0}(x)+%
{\displaystyle\int\nolimits_{0}^{t}}
Z_{t-s}(x)f(x,s)ds
\]
for all $t\in\lbrack0,T]$. Furthermore,
\[%
\begin{array}
[c]{ccc}%
C_{0}(%
\mathbb{Q}
_{p}^{n},\mathbb{R}) & \longrightarrow & C_{0}(%
\mathbb{Q}
_{p}^{n},\mathbb{R})\\
&  & \\
h & \longmapsto & Z_{t}(x)\ast h(x),
\end{array}
\]
for $t\geq0,$ gives rise to a Feller semigroup.
\end{theorem}

\begin{proof}
Set
\[
(\digamma(t)u)(x)=Z_{t}(x)\ast u(x)\text{, for }t\geq0\text{, }u\in
\mathcal{D}_{%
\mathbb{R}
}\left(
\mathbb{Q}
_{p}^{n}\right)  .
\]
By using Lemmas \ref{Lemma 8}, \ref{Lemma 9},
\[
T(t)\mid_{\mathcal{D}_{%
\mathbb{R}
}\left(
\mathbb{Q}
_{p}^{n}\right)  }=\digamma(t)\mid_{\mathcal{D}_{%
\mathbb{R}
}\left(
\mathbb{Q}
_{p}^{n}\right)  }\text{for }t\geq0\text{.}%
\]
Now, since
\[
\left\Vert \digamma(t)u\right\Vert _{L^{\infty}}=||Z_{t}\ast u||_{L^{\infty}%
}\leq||Z_{t}||_{M}||u||_{L^{\infty}},
\]
where $||Z_{t}||_{M}$ denotes the total variation of the finite Borel measure
$Z_{t}$, and since $\mathcal{D}_{%
\mathbb{R}
}\left(
\mathbb{Q}
_{p}^{n}\right)  $ is dense in $C_{0}(%
\mathbb{Q}
_{p}^{n},\mathbb{R})$, we conclude $T(t)=\digamma(t)$ for $t\geq0$. Finally,
by Theorem \ref{Feller semigroups}, $\digamma(t)$ gives rise to a Feller semigroup.
\end{proof}

\begin{remark}
\label{Levy process}With the hypotheses of Theorem \ref{Feller semigroups} and
assuming that $p(0)=0$, we obtain the existence of a L\'{e}vy process
$(X_{t})_{t\geq0}$ with state space $\mathbf{%
\mathbb{Q}
}_{p}^{n}$, such that $Z_{t}(x)=P_{X_{t}-X_{0}}(x),$ where $P_{X_{t}-X_{0}}$
denotes the distribution of the random variable $X_{t}-X_{0}.$ This result
follows from \cite[Section 2]{Evans}, since $Z_{t}$ is a convolution semigroup
such that $Z_{t}\rightarrow\delta_{0}$ weakly as $t\rightarrow0^{+}$.
\end{remark}


\bigskip

\begin{thebibliography}{99}                                                                                               %


\bibitem {Alberio et al}Albeverio S., Khrennikov A. Yu., Shelkovich V. M.,
Theory of $p$-adic distributions: linear and nonlinear models. London
Mathematical Society Lecture Note Series, 370. Cambridge University Press,
Cambridge, 2010.

\bibitem {Av-4}Avetisov V. A., Bikulov A. Kh., Osipov V. A., $p$-adic
description of characteristic relaxation in complex systems, J. Phys. A 36
(2003), no. 15, 4239--4246.

\bibitem {Av-5}Avetisov V. A., Bikulov A. H., Kozyrev S. V., Osipov V. A., $p
$-adic models of ultrametric diffusion constrained by hierarchical energy
landscapes, J. Phys. A 35 (2002), no. 2, 177--189.

\bibitem {Berg-Gunnar}Berg Christian, Forst Gunnar, Potential theory on
locally compact abelian groups. Springer-Verlag, New York-Heidelberg, 1975.

\bibitem {C-H}Cazenave Thierry, Haraux Alain, An introduction to semilinear
evolution equations. Oxford University Press, 1998.

\bibitem {Courrege}Courr\`{e}ge Ph., Sur la forme
int\'{e}gro-diff\'{e}rentielle des op\'{e}rateurs de $C_{k}^{\infty}$ dans $C$
satisfaisant au principe du maximum. S\'{e}minaire Brelot-Choquet-Deny.
Th\'{e}orie du potentiel, tome 10, no 1 (1965-1966), exp. no 2, p. 1-38.

\bibitem {Ch-Z-2}Chac\'{o}n-Cortes L. F., Z\'{u}\~{n}iga-Galindo W. A.,
Non-local operators, non-Archimedean parabolic-type equations with variable
coefficients and Markov processes, Publ. Res. Inst. Math. Sci. 51 (2015), no.
2, 289--317.

\bibitem {Ch-Z-1}Chac\'{o}n-Cortes L. F., Z\'{u}\~{n}iga-Galindo W. A.,
Nonlocal operators, parabolic-type equations, and ultrametric random walks. J.
Math. Phys. 54 (2013), no. 11, 113503, 17 pp. Erratum 55 (2014), no. 10,
109901, 1 pp.

\bibitem {E-K}Ethier Stewart N., Kurtz Thomas G., Markov Processes -
Characterization and convergence, Wiley Series in Probability and Mathematical
Statistics, John Wiley $\And$ Sons, New York, 1986.

\bibitem {Evans}Evans Steven N., Local Properties of L\'{e}vy processes on a
totally disconnected group, J. Theoret. Probab. 2 (1989), no. 2, 209--259.

\bibitem {Gelfand}Gel'fand, I. M., Vilenkin, N. Ya., Generalized Functions.
Vol 4. Applications of Harmonic Analysis. AMS Chelsea publishing, 2010.

\bibitem {Jacob-vol-1}Jacob N., Pseudo differential operators and Markov
processes. Vol. I. Fourier analysis and semigroups. Imperial College Press,
London, 2001.

\bibitem {Jacob-vol-2}Jacob N., Pseudo differential operators and Markov
processes. Vol. II. Generators and their potential theory. Imperial College
Press, London, 2002.

\bibitem {Jacob-vol-3}Jacob N., Pseudo differential operators and Markov
processes. Vol. III. Markov processes and applications. Imperial College
Press, London, 2005.

\bibitem {Koch}Kochubei Anatoly N., Pseudo-differential equations and
stochastics over non-Archimedean fields. Marcel Dekker, Inc., New York, 2001.

\bibitem {Kozyrev  SV}Kozyrev S. V., Methods and Applications of Ultrametric
and $p$-Adic Analysis: From Wavelet Theory to Biophysics, Sovrem. Probl. Mat.,
12, Steklov Math. Inst., RAS, Moscow, 2008, 3--168.

\bibitem {Obata}Obata Nobuaki, White noise calculus and Fock space. Lecture
Notes in Mathematics, 1957. Springer-Verlag, 1994.

\bibitem {R-Zu}Rodr\'{\i}guez-Vega J. J., Z\'{u}\~{n}iga-Galindo W. A.,
Taibleson operators, $p$-adic parabolic equations and ultrametric diffusion,
Pacific J. Math. 237 (2008), no. 2, 327--347.

\bibitem {Taibleson}Taibleson M. H., Fourier analysis on local fields.
Princeton University Press, 1975.

\bibitem {Taira}Taira Kazuaki, Boundary value problems and Markov processes.
Second edition. Lecture Notes in Mathematics, 1499. Springer-Verlag, 2009.

\bibitem {To-Z}Torresblanca-Badillo A., Z\'{u}\~{n}iga-Galindo W. A.,
Ultrametric Diffusion, exponential landscapes, and the first passage time
problem, arXiv: 1511.08757v2 [math-ph] 21 Jun 2016

\bibitem {Va1}Varadarajan V. S., Path integrals for a class of $p$-adic
Schr\"{o}dinger equations, Lett. Math. Phys. 39 (1997), no. 2, 97--106.

\bibitem {V-V-Z}Vladimirov V. S., Volovich I. V., Zelenov E. I., $p$-adic
analysis and mathematical physics. World Scientific, 1994.

\bibitem {Zu0}Z\'{u}\~{n}iga-Galindo W. A., Local Zeta Functions,
Pseudodifferential operators, and Sobolev-type spaces over non-Archimedean
local Fields. arXiv:1704.07965.

\bibitem {Zuniga-LNM-2016}W. A. Z\'{u}\~{n}iga-Galindo, Pseudodifferential
equations over non-Archimedean spaces. Lectures Notes in Mathematics 2174,
Springer, 2016.

\bibitem {Zu1}Z\'{u}\~{n}iga-Galindo W. A., Non-Archimedean White Noise,
Pseudodifferential Stochastic Equations, and Massive Euclidean Fields, J.
Fourier Anal. Appl. 23 (2017), no. 2, 288--323.

\bibitem {Z1}Z\'{u}\~{n}iga-Galindo W. A., The non-Archimedean stochastic heat
equation driven by Gaussian noise, J. Fourier Anal. Appl. 21 (2015), no. 3, 600--627.
\end{thebibliography}
\end{document}